\documentclass[11pt, a4paper, twoside]{amsart}
\usepackage{a4wide}

\usepackage[utf8]{inputenc}
\usepackage[english]{babel} %tavutus
\usepackage{amsthm}
\usepackage{amsmath}
\usepackage{mathtools}
\usepackage{amssymb}
\usepackage[all]{xy}
\usepackage{tikz}
\usepackage{enumerate} % (i) (ii) (iii) etc. to enumerate 

\theoremstyle{plain}
\newtheorem{theorem}{Theorem}[section]
\theoremstyle{definition}

\theoremstyle{remark}

\theoremstyle{plain}

\theoremstyle{plain}
\newtheorem{proposition}{Proposition}[section]
\theoremstyle{plain}
\newtheorem{corollary}{Corollary}[section]
\theoremstyle{plain}

\def\bt{\boldsymbol\theta}
\def\btk{\boldsymbol\theta}
\def\bxk{{\bf x}}
\def\tr{{\rm Tr}}
\def\R{\mathbb{R}}
\def\S{\mathbb{S}}

\def\Timk{\prod_{j=1}^k}

\makeindex

\usepackage{hyperref}
\theoremstyle{plain}
\newtheorem{lemma}{Lemma}[section]

\begin{document}

 \title[Inverse scattering and tomography on probability distributions]{Correlation imaging in inverse scattering is tomography on probability distributions}
 \author{Pedro Caro, Tapio Helin, Antti Kujanp\"a\"a and Matti Lassas}
   \begin{abstract} 
   Scattering from a non-smooth random field on the time domain is studied 
    for plane waves that propagate simultaneously through the potential in variable angles.
   %The setup consists of a superposition of plane waves which propagate through the potential in different angles. 
    We first derive sufficient conditions for stochastic moments of the field to be recovered from correlations between amplitude measurements of the leading singularities, detected in the exterior of a region where the potential is almost surely supported.
    The result is then applied to show that if two sufficiently regular random fields yield the same data, they have identical laws as function-valued random variables.
   % Due to earlier development in the field of moment problems, our result implies that the data uniquely determines the related probability distribution for many classes of random fields, including in particular compactly supported %$H^2$-potentials whose probability distribution is exponentially finite.
   \end{abstract}
 \maketitle

\section{Introduction}
Randomness is often an inherent part of any computational model for an applied inverse problem.
For instance, it can reflect the chaotic evolution of the system or the perspective that the unknown object of interest is rough and vastly complex.
Ultimately, the observational noise is most often probabilistic in nature.
If the statistics of the system can be described to a good approximation, 
it can be desirable to transform the problem paradigm by considering correlations or other statistical moments of the data distribution and how that information 
relates to the relevant system parameters.
This approach is often called \emph{correlation based imaging} in literature and it has been recently studied for a variety of inverse problems (see e.g. applications in seismic imaging \cite{garnier2016passive}).
%Since any practical solution to inverse problems is limited in accuracy by the combined effect of ill-posedness and observational noise,
%correlation based imaging is attractive in applications, where the observational data is extensive but exceptionally corrupted or noisy.

In this paper we consider scattering of waves from a time-independent random potential $V$ supported on a fixed compact set in $\R^n$. 
We study an inverse problem 
of recovering the law of $V$ given certain correlation data in the exterior of the potential.
The wave propagation is governed by 
\begin{eqnarray}\label{eq:main_problem}
 \left(\Box  - V(x) \right) u  (x,t,\tilde\bt) & = &  0,\nonumber\\
u (x,t,\tilde\bt) &  = & \sum_{j=1}^N \delta \left(t- x \cdot \tilde\theta_j \right) + u_{\text{sc}}(x,t,\tilde\bt), \\
u_{\text{sc}} (x,t,\tilde\bt)  & = & 0, \quad \text{for } t \ll 0,\nonumber 
\end{eqnarray}
where $(x,t) \in  \R^{n+1}$, $\tilde\bt= \tilde\bt^N := (\tilde\theta_1,\dots,\tilde\theta_N)$, $\Box := \partial_t^2 - \Delta$ is the wave operator, and the potential 
\[
V : \R^n \times \Omega \rightarrow \R, \ V(x) = V(x, \omega)
\]
is a random generalized function, that is, a measurable map from the probability space $\Omega$ into a linear subspace of generalized functions which, in this paper, will be contained in the Sobolev space $H^2( \R^n ):= W^{2,2}( \R^n )$ endowed with the Borel $\sigma$-algebra. 
Above, the incoming wave is given by a superposition of $N$ plane waves.
We shall omit the parameter $\omega \in \Omega$ from notation and write $V(x)$ instead of $V(x,\omega)$. 
%We will often omit the variable $\omega$ from the notation. 
%The model 

Given a family of directions $\theta_1, \dots, \theta_k \in \mathbb{S}^{n-1}$ which are not necessarily distinct from each other, let us denote the trajectory of the plane wave $\delta (t - \theta_j \cdot x )$, $j\in \{ 1,\dots,k \}$ by 
\begin{equation*}
\Sigma_j := \{ (x,t) \in \R^{n}\times \R : x \cdot \theta_j = t \} \text{. } 
\end{equation*}
%for $j=1,\dots,k$, where $\theta_j \in \S^{n-1}$. 
In the following $\Sigma_j(t)$ stands for the $(n-1)$-dimensional hyperplane in $\Sigma_j$ given a fixed $t\in\R$. It is well-known that after suitable time $T>0$ the potential $V$ produces a discontinuity in the scattered field across the hyperplane $\Sigma_j(t)$ for $t>T$. More importantly, the discontinuities carry 
information regarding the integral values of $V$ over the bicharacteristic lines.
%The plane waves 
 By setting measurement devices outside the region where the potential is almost surely supported one captures ``the shadows'' of the potential in different angles from the outgoing wave fronts as they collide into the detectors. %Such a setup could be a model for free particles, say photons, .
% a single measurement of the such form would be unreliable and difficult to perform in practise but  
After repeating the observation one can consider empirical correlations between patterns captured by separate measurement devices. %Note that 
%There is also a corresponding counterpart
%One can also consider the model on frequency domain by taking Fourier transform with respect to time, hence transforming the wave equation into Helmholtz equation. 
%The singularities are then transformed into decay properties 

Motivated by this line of thought, we study correlations in the exterior of the potential. 
%More precisely, 
We assume that the observations approximate the following generalised function to a good approximation. We will refer to it as the exterior data:
\begin{equation}
	\label{eq:data}
(\bxk, \btk) \mapsto D^k(\bxk,\btk)=D_V^k(\bxk,\btk) \quad \textrm{and for all} \; (\bxk,\btk)  \in \Timk (\R^n \times \S^n),
\end{equation}
where $\bxk = (x_1,\ldots,x_k)$, $\btk = (\theta_1,\ldots,\theta_k) $ and
\[
D^k (\bxk,\btk) := 
%D_V^k (\bxk,\btk)  :=
 \mathbb{E} \left( \lim_{s\rightarrow \infty}   \prod_{j=1}^k [ u_{sc}   ]_{\Sigma_j}(x_j + s\theta_j ,  \tilde\bt) \right).
\]
%Up to residual in $H^1_{loc}$ one can  
The notation $[ u_{sc}  ]_{\Sigma_j}  (x, \tilde\bt)$ stands for the jump across $\Sigma_j$,
\[
[ u_{sc} ]_{\Sigma_j} (x, \tilde\bt)=( \tr_{\Sigma_j}^+  -  \tr_{\Sigma_j}^-) u_{sc} (x, x \cdot \theta_j, \tilde\bt ),
\]
where $\tilde\bt  \in \prod_{j=1}^{N} \mathbb{S}^{n-1}$, $N = \# \cup_{j=1}^k \{  \theta_j\}$ is a parametrisation of distinct elements in $ \theta_1, \dots,\theta_k $, that is, a bijection $j \mapsto \tilde\bt_j$ from $\{1,\dots,N\}$ into $\bigcup_{j=1}^k \{ \theta_j \}$. 
%\[
%\{1,\dots,N\} \rightarrow  \bigcup_{j=1}^k \{ \theta_j \} : j \mapsto \tilde\bt_j
%\]
%is bijection. 
 The traces  
%\[
$
\tr_{\Sigma_j}^{\pm} 
$ 
%: H^1_{loc} \big( \{ \R^n \setminus \overline{B(0,R)} \}  \times \R \big) \rightarrow H^{1/2}_{loc} \big( \Gamma_j \cap ( \{  \R^n \setminus \overline{B(0,R)} \} \times \R ) \big) 
% \]
 stand for restrictions to the boundary $\Sigma_j$ from the
upper and lower half-spaces 
\[
\{(x,t) \in \R^{n+1} : 0 \leq \pm (t - x\cdot \theta_j) , \ x\in \R^n \setminus \overline{B(0,R)} \}
\]
 and the radius $R$ is chosen large enough so that the support of the potential is almost surely contained in the ball $B(0,R)$.  
%$B(0,r)$ and the trajectories of the plane waves outside the ball intersect in a region where the leading singularities of the scattered wave vanish.
The angles $\tilde\theta_j$, $j=1,\dots, N $ of the incoming waves $\delta ( t- x\cdot \tilde\theta_j )$ are considered to be distinct from each other for technical reasons. 
The vectors $\theta_j$, $j=1,\dots, k $ are not required to be distinct. 
We give a precise definition of the the restrictions in the beginning of the next section.

The exterior data is invariant with respect to  
%to permutations $ ( \theta_1, \dots, \theta_k ) \mapsto ( \theta_{i_1}, \dots, \theta_{i_k} ) $, 
shifts along the trajectories:
\begin{equation}\label{shift}
D^k (\bxk, \btk ) = D^k (x_1 + s_1 \theta_1 , \dots,x_k + s_k \theta_k , \btk ), \ s_1, \dots, s_k \in \R
\end{equation} Therefore, without loss of information, it can alternatively be given on the tangent bundle $\prod_{j=1}^k T  \mathbb{S}^{n-1}$ by identifying each fiber $  T_{\theta} \mathbb{S}^{n-1}$, with the orthogonal complement  $ \{\theta\}^\perp \times \{ \theta \} \subset \R^n \times \{ \theta\} = T_\theta \R^n $, $\{ \theta \}^\perp := \{ x \in \R^n : x \cdot \theta = 0 \}$ (see Appendix).
%For each base point $\btk \in \prod_{j=1}^k \mathbb{S}^{n-1}$ 
The data is then expressed as a generalised function
\[
 \btk \mapsto D^k( \btk) \in \mathcal{D}' \left(T_{\btk} \prod_{j=1}^k   \mathbb{S}^{n-1}\right), \  \btk \in \prod_{j=1}^k \mathbb{S}^{n-1},
%\prod_{j=1}^k \mathbb{S}^{n-1} \rightarrow   \mathcal{D}' ( \prod_{j=1}^k T  \mathbb{S}^{n-1}),
\]
given by
\[
D^k ( \bt ) (y_1, \dots,y_k ) =  D^k  \big( (y_1,0) , \dots, (y_k,0)     , \bt \big)  
\]
in coordinates $y_j = (y_j^1, \dots, y_j^{n-1}) \in  \R^{n-1}$ of $T_{\theta_j}  \S^{n-1}$, $j=1,\dots,k$. 
%, where 
%\[
%(y_j,0) = (y_j^1, \dots, y_j^{n-1} , 0 ) \in \{ \theta_j \}^\perp, \ j=1,\dots,k.
%\] 
For a smooth potential one obtains more practical form, $D^k \in C^\infty (\prod_{j=1}^k T\mathbb{S}^{n-1})$ which is applied in Proposition \ref{pro12378}. 
In comparison with the typical far-field measurement 
%---which in the direction of an incident plane wave collapses into an integral of the potential over $\R^n$ (see \cite{MR3224125})---
the exterior data contains additional tangential parameters $(y_1, \dots,y_k )$.
The data models measurements of statistical correlation between amplitude peaks of leading singularities which scatter from randomly varying objects as a result of interaction between the associated potential and the incident waves.
The tangential parameters correspond to points in the surface of a detector plate or a film. 
The interpretation requires that during a single measurement the random state of the target varies slowly compared to the speed of propagation of waves (e.g. the speed of light or sound).
%The propagation speed of waves is assumed to be sufficiently high if the potential varies in time.
 %The data corresponds to the leading term of the progressive wave expansion (see \cite{MR3224125}) which is 
%The far-field fails to capture enough information in the direction of propagation of an incident plane wave due to blow-up of the associated oscillatory integral.   
%Integration with respect the parameter essentially leads to the mean value of the far-field in the forward direction

%%Alternatively, one can consider the model on frequency domain by taking Fourier transform with respect to time, hence changing the wave equation into the Helmholtz equation. 
%An equivalent model 
%The system then takes the form of stationary quantum potential scattering, 
%\[
%\big(\Delta + k^2 +V(x)\big) \hat{u}(x,k,\tilde\bt) = 0
%\]
%with incident wave $\sum_{j=1}^N e^{-i k x \cdot \tilde\theta_j} $ and Sommerfeld radiation condition. 
%It is direct consequence of our analysis  for compactly supported $V \in H^2( \R^n ) \cap L^\infty (\R^n)$ the  %the leading singularity in the scattered wave is of the order $-1$ and propagates along bicharacteristic lines.
%%The singularities are then transformed into decay properties 
%%As the Fourier transform 

The model (\ref{eq:main_problem}) can also be interpreted in the frequency domain. 
Taking Fourier transform with respect to time converts the model into a time-harmonic system,  
\[
(\Delta + \lambda^2 + V(x)) \widehat{u} (x,\lambda,\tilde\bt) = 0,
\]
associated with angular dependency $\lambda^2$, an electric potential $V$,
%\[
%(\Delta + \lambda^2 + V(x)) \widehat{u} (x,\lambda,\tilde\bt) = 0
%\]
the incident wave $\widehat{u}_I (x, \lambda, \tilde\bt)= \sum_{j=1}^N e^{-i \lambda x \cdot \tilde\theta_j}$ and the Sommerfeld radiation condition. In quantum mechanics such a model typically arises from the time-dependent Schr\"odinger equation,
\[
\big( V(x) - \Delta \big) \Psi (x,s, \tilde\bt)  = i\frac{2m}{\hbar}  \partial_s  \Psi (x,s, \tilde\bt).  
\]
More precisely, the stationary wave $\widehat{u} (x,\lambda,\tilde\bt)$ equals the spatial part $ \Psi (x,0, \tilde\bt, \lambda)$ of the solution $ \Psi (x,s, \tilde\bt, \lambda) =  \Psi (x,0, \tilde\bt, \lambda)e^{-i \lambda s / \hbar} $ with energy $E =
%E \propto  
\lambda^2$ and the general solution is a superposition of these waves. 
Singularities of a scattered wave  in the time domain appear in the frequency domain as slower decay of the Fourier transform in directions that belong to the wave front set. The leading singularity of $u_{sc} (x,t,\tilde\bt)$ in time corresponds to the high frequency asymptote of $\widehat{u}_{sc} (x,\lambda,\tilde\bt) =   \Psi (x,0, \tilde\bt, \lambda) - \widehat{u}_I (x, \lambda, \tilde\bt) $, that is, the first term 
%$b_{s_0}$
 of an asymptotic series expansion 
 %\sim \sum_{m=0}^\infty b_{s-m} (x,k, \tilde\bt)$, $b_{s-m} (x,k, \tilde\bt) $ 
 with respect to the variable $\lambda$.
%\[
%\widehat{u}_{sc} (x,k,\tilde\bt) = \sum_{m=s_0}^{s} b_m(x, \tilde\bt) k^{-m}  + O(k^{-s-1} ).
%\]
%As obtained from the solution constructed in 
The exterior data for a potential in $ C_c^{\infty} (\R^{n})$ with the support almost surely contained in a ball $B(0,R)$ is the time domain counterpart of 
\[
\begin{split}
%D(\bxk,\bt) \simeq 
\mathbb{E} \left( c_{k,n}  \Timk   \lim_{\lambda \rightarrow \infty}   e^{-i \lambda   x_j  \cdot \theta_j   } \lambda    \Psi_{sc} (x_j,0, \tilde\bt, \lambda) \right) ,
%=\mathbb{E} \left( c_{k,n}  \Timk   \lim_{\lambda \rightarrow \infty}   e^{-i \lambda x_j  \cdot \theta_j   } \lambda  \widehat{u}_{sc} (x_j  , \lambda, \tilde\bt ) \right), 
 \ \ \bxk \in  \Timk \R^n \setminus \overline{B(0,R)}, \ \ \lambda \in \R \setminus \{0 \} %\text{ for }\bxk \in \Timk \R^n \setminus \overline{B(0,R)}. 
\end{split}
\]
%for $\bxk \in  \Timk \R^n $.
%for $R>0$ large enough to satisfy. 
which can be obtained directly from the progressive wave expansion \cite{MR3224125}.  
The exterior data therefore describes statistical correlations in scattering patterns produced by multiple high-energy particles interacting with the random potential. 
Similar interpretation of data is expected to be valid for potentials with less regularity but will not be studied here.

\subsection{The Results}

 For a random field
 \[
 V : \mathbb{R}^n \times \Omega \rightarrow \R : (x,\omega) \mapsto V (x , \omega)
  \]
we introduce the following three conditions:
%\begin{equation}\label{qoo1}
\begin{enumerate}%[i)]
\item[(C1)] $V \in H^2( \R^n ) \cap L^\infty (\R^n)$ almost surely
%for almost every $\omega \in \Omega$,
\item[(C2)] There is a compact set $K\subset \R^n$ such that $V $ is supported in $K$ almost surely
%for almost every $\omega \in \Omega$
\item[(C3)] There is a constant $a>0$ such that $\mathbb{E}e^{a\|V\|_{H^2} } < \infty$, 
%\end{equation}
   %
%and are almost surely compactly supported in $K$. 
\end{enumerate}
%An example of a function that satisfies (C1) 
The compact set $K$ in (C2) can be replaced by a ball. 
 A field with almost surely bounded $H^2$-norm, $ \|V \|_{H^2} \in  L^\infty (\Omega) $, satisfies (C3).

Our main results are given by the following theorems.
\begin{theorem}\label{thm:main}
%Let the potential $V$ in \eqref{eq:main_problem} be a random field on $\R^n$ supported on a compact set $\dom \subset \R^n$ and suppose that $V \in H^2 (\R^n)$
% almost surely. 
Let $V : \mathbb{R}^n \times \Omega \rightarrow \R$ be a random field that satisfies the conditions \textup{(C1)} and \textup{(C2)} above. 
Then the $k^{\text{th}}$ moment map $M^k \in \mathcal{E}' \left(\Timk \R^{n}\right)$, given by
\begin{equation}
	\label{eq:moment_map}
 M^k (\bxk) = \mathbb{E} \left( \Timk V (x_j) \right),
\end{equation}
is uniquely determined by the exterior data \eqref{eq:data} for any $k\in \mathbb{N}$.
\end{theorem}

In particular, the proof of Theorem \ref{thm:main} provides a reconstruction strategy to recover an explicitly defined sequence $M^k_\epsilon$, $\epsilon>0$ of smooth functions that converges to $M^k$. Moreover, as Gaussian random fields are determined by their mean field and covariance function, the Theorem \ref{thm:main} yields an immediate corollary for this 
important class of random models.

\begin{corollary}
If the potential $V$ in Theorem \ref{thm:main} is a Gaussian random field, then the probability distribution of $V$ is uniquely determined given the exterior data $$(\bxk,\btk) \mapsto D^j (\bxk,\btk) \quad \textrm{for} \; j=1,2, \textrm{and for all} \; (\bxk,\btk)  \in \Timk (\R^n \times \S^n).$$
\end{corollary}

Theorem \ref{thm:main} can be applied to derive sufficient conditions for the associated laws $V_*\mathbb{P} : A \mapsto \mathbb{P} \{ \omega \in \Omega :  V(\cdot, \omega) \in A\}$ in $H^2(\mathbb{R}^n)$ to be unique, i.e., the only positive Borel measure related to the data:
\begin{theorem}\label{mainmaintheorem}
Let $V$ and $W $ be two 
%$H^2 
%(\mathbb{R}^n)$ -valued 
random fields that satisfy the conditions \textup{(C1), (C2), (C3)} and yield the same exterior data 
%\[
%D_V^k  = D_W^k 
%\]
 for every $k \in \mathbb{N}$.
Then the potentials have the same laws (i.e. probability distributions):
\[
V_*\mathbb{P} = W_*\mathbb{P}.
\]
\end{theorem}

%In many simple cases the answer is positive, as shown in the Appendix.
The proof of Theorem \ref{mainmaintheorem} relies partly on a result from \cite{DVURECENSKIJ200255} considering determinateness for Euclidean multivariate moment problem. This is the main reason to consider fields of the form (C3). 
% the condition 
%[Theorem 2.1., Corollary 1.]{DVURECENSKIJ200255}, the condition 
%\begin{equation}\label{suomiequaatio}
%\int\limits_{\R^k} e^{a | y| } d \mu (y) < \infty,
%\end{equation}
%known as the exponentially bounded
%for some $a>0$  implies that the set of polynomials, that is, the linear combinations of the moments, is dense in $L^p (\R^k, \mu)$, $p\geq 1$, thus implying determinateness of the measure. 
%In addition, 

\subsection{Previous literature}

A natural path for acquiring the spatial correlation data in practise is averaging a large number of independent observations of the scattered field in time.
Similar problem setup often appears in wave and particle propagation in heterogeneous media. Typically, heterogeneous medium is modelled as a realization of random field with a priori known statistics. In literature, multi-scale analysis or homogenization is often utilized with the aim of capturing the effective properties of the propagation. Notice that our work does not assume any scale separation. We refer to the articles \cite{borcea2003, bal2010kinetic, ishimaru1978wave, fouque2007wave, dehoop09} for various perspectives on wave propagation (whether classical or quantum) in random media.
%Recently, inverse problems related to imaging of random media have received wide attention \cite{bal2002, bal2005time,bal2007,bal2003, borcea2011, borcea2015, , borcea2016,dehoop2012, , }.

The origin of the randomness can also be a specific source in the considered system, see e.g. the early work \cite{devaney1979inverse} on inverse random source problems.
Since then correlation based imaging in random source problems have been considered widely 
in the framework of different PDE models by Li, Bao and others \cite{li2017stability, bao2017inverse, li2017inverse, Bao2016several, Bao2014Helmholtz, Li2011source}.
Other applications include telescope imaging \cite{helin2018atmospheric} and seismic imaging \cite{garnier2009passive, garnier2016passive,MR2482156, MR3318382, helin2016correlation}. 
Imaging in random media has also been studied by Borcea and others \cite{MR3606418, MR2249471, MR1943391, MR3470117}, and for backscattering by Shevtsov \cite{MR1708341}.

Our paper provides continuation to previous work by the authors in \cite{LPS, HLP, caro2016inverse}, where the averaging procedure to estimate correlations 
is based on a single realization of the observed data, i.e.,
the random potential or boundary condition is sampled only once. Such an approach can reveal valuable information of the leading order statistics of the unknown field. However,
the full probability distribution of the unknown is not recovered unlike here.

The literature on inverse scattering for deterministic potentials is rather wide and we cite here only a few works in the field. In \cite{zbMATH00939857}, Colton and Kirsch introduced the linear sampling method to determine the support of an imperfect conductor given the far-filed of the scattered wave. Uniqueness for the inverse acoustic medium problem was proved by Nachman \cite{zbMATH04105476}, Novikov \cite{zbMATH04129351}, and Ramm \cite{zbMATH04028038}.
Uniqueness for the inverse backscattering problem in a generic class of potentials was proved by Eskin and Ralston \cite{MR1012864, MR1110451}. 
Uniqueness for angularly controlled potentials has been proved by Rakesh and Uhlmann  
\cite{MR3224125}. Single measurement inverse problems for the wave equation is explored by Rakesh \cite{MR2384771} and by Liu and others \cite{MR3325344, MR2582602}. 
%Earlier partial results for the inverse backscattering problem  for Schr\"odinger equation has been obtained in \cite{Ike,LL,MU,Rakesh1,Ste1, MR1082237,W1}. Approximative or numerical reconstructions have been studied in \cite{Bei,Ike}.
%The recovery of singularities of the potential from backscattering data is analyzed in
%\cite{B1,MR1243710,OPS,MR2309667,R1, RV,Ser1,Ser3,Ser2}. %\HOX{The references need to be organised by numbering.}
%Other references on inverse backscattering for a time-harmonic Schr\"{o}dinger equation are \cite{MR2512860, MR2781141,U1}.
%The backscattering problem has also been studied in the framework of acoustic scattering (see \cite{MR1466676,MR1614940,MR1607660}) and Maxwell equations (see \cite{MR1648523}). For a concise treatment of classical inverse scattering, we refer to \cite{ColtonKress}.
Use of moments in inverse problems for partial differential equations has previously been studied by Kurylev and others in \cite{MR1474373, MR1991787,  MR1491678}.

Finally, we want to point out that there exists a variety of criteria for the moment problem to be determinate. These might provide potential alternatives for applications, where the exponential moment is not bounded. This aspect is not the focus of our work and we refer to \cite{schmudgen2017moment} and references therein regarding such generalizations.
%\cite{putinar2008multivariate, putinar2006multivariate, MathScand14406, 10.2307/24491444, Schmudgen1991, DVURECENSKIJ200255, schmudgen2017moment, MR0121660, MR693807, MR0194899}.

%Although some of the results may be challenging from the practical point of view, 
%for more studies on the subject.  
%For example, in \cite[Theorem 6]{putinar2008multivariate} it is shown that the condition
%\begin{equation}\label{2936377337}
%\sum_{k=0}^\infty \big[ \sum_{| \alpha | = k} | m_{\alpha} |^2 \big]^{-1/2k}  = \infty,
%\end{equation}
%$m_\alpha :=\mathbb{E} ( V_1^{\alpha_1}  \dots V_n^{\alpha_n} )$, $\alpha \in \mathbb{Z}^n_+$ is sufficient, thus implying the following corollary:

%\begin{corollary}\label{asdj261376123}
%Let the moments $m_\alpha$, $\alpha \in \mathbb{Z}^n_+$,  be finite real numbers at $x_1, \dots, x_n \in \R^n$ and assume that the condition
% % \[
% % \liminf_{p \rightarrow \infty} \Big( \sum_{| \alpha | = p} \frac{m_{2\alpha}^n }{\alpha !}  \Big)^{1/p} < \infty
% % \]
% \eqref{2936377337} 
%  holds. Then, 
%under the conditions of Theorem \ref{thm:main}, the measure $\mu (V_1, \dots, V_n)$ is uniquely described by the data.   
%\end{corollary}

 \section{Proofs of the Results}\label{324762384234239048}

\subsection{Preliminary Definitions}
Let us introduce some relevant notations and definitions. % \comment{What is $X$? Perhaps change to $\R^d$?}
Given a distribution  $v\in \mathcal{D}'(X)$ on a smooth manifold $X$ let $WF (v)$ be the wave front set of $v$, i.e., the complement of the collection of co-vectors $(z_0,\xi_0) \in X \times( \R^{\dim (X)} \setminus \{0\} ) $ such that  in some neighbourhoods $U\ni z_0$ and $V\ni \xi_0$ the decay estimate 
\[
 \widehat{\varphi v} (\tau \xi) =O (\tau^{-m}), \text{ for } \tau \rightarrow \infty,\text{ uniformly in }\xi \in V,
\]
holds for every $\varphi \in C_c^\infty(U)$ and $m\in \mathbb{N}$. 
Let $\Gamma$ be a closed cone in $T^* (X)$ and define $\mathcal{D}_\Gamma' (X)$ as the collection of distributions $v\in \mathcal{D}'(X )$ such that $WF(v) \subset \Gamma$. Similarly, we define 
\begin{equation*}
	H^s_{\Gamma, loc}(X):= H^s_{loc}(X) \cap  \mathcal{D}_\Gamma'(X )
\end{equation*}
for $s\in \R$.
%on some conical neighbourhood $V$ of $\xi_0$. %As usual, we denote $WF (v)=\cup_{s\in \R} WF_s (v)$.
Given a submanifold $Y \subset X$ we denote by $N^* Y $ the conormal bundle of $Y$, that is, the collection of vectors $\xi \in T^*X$ such that $\langle \xi ,v \rangle = 0$ for every $v \in T_{\pi(\xi)} Y$, where $\pi : T^*X \rightarrow X$ stands for the bundle projection. 
Given a distribution $v$ on $\R^{n+1}$ with $WF (v) \cap N^* \Sigma_i = \emptyset$, the trace ${\rm Tr}_{\Sigma_i} (v) \in \mathcal{D}'(\Sigma_i) $ is well defined and depends continuously on $v$ with respect to the topology of distributions (see \cite{duistermaat2010fourier}). 
%Let $\Gamma$ be a closed cone in $T^* (X)$ and define $\mathcal{D}_\Gamma' (X)$ as the collection of distributions $v\in \mathcal{D}'(X )$ such that $WF(v) \in \Gamma$. Similarly, $H^s_{loc,\Gamma}(X):= H^s_{loc}(X) \cap  \mathcal{D}_\Gamma'(X )$, for $s\in \R$. 
%In the space $\mathcal{D}_\Gamma' (X) \times H^1_{loc} (X)$ 
%we define the restriction to a submanifold $M\subset X$ with $\Gamma \cap N^*M = \emptyset$, $\dim(M) = \dim (X) -1$, by $\tr_{M} (v_1,v_2) :=  \iota_{M}^*(v_1) + \tr_{M}(v_2)$ where 
%\[
% \tr_{M} : H^1_{loc}(X) \rightarrow H^{1/2}_{loc}(M). 
% \]
% stands for the typical trace operator.
 %We shall slightly generalise the definition.  
%  Let $\{ M_s \}_{s \in \R}$ be a foliation of the manifold $X$ into sub-manifolds $M_s = F_s(M_0)$ of codimension 1 with respect to a smooth non-vanishing flow $F :\R \times  X \rightarrow X$, $F_s := F(s,\cdot )$, and assume that 
% $N^*M_s \cap \Gamma = \emptyset$ for the parameters $s  \neq 0$ in some interval $(-\epsilon,\epsilon)$ around origin. We define 
% \[
% \tr_{M_0}^\pm (v_1,v_2) := \lim_{s \rightarrow 0\pm} \tr_{M_0} ( F_s^* v_1,v_2)
% \]
% whenever the limits exist in $\mathcal{D}' (M_0)$. 
% %We shall denote $ \tr_{\Sigma_j} v := \tr_{\Sigma_j} (v_1,v_2)$ if the decomposition $v= v_1+v_2$ is clear from the context. 

%Fix $M_0 = \Sigma_j$, $F_s (x,t) := (x,t+s)$, for $(x,t) \in \R^{n+1}$ and $s \in \R$.
It is shown in the next section that $u_{sc}$ can be split into two parts $u_{sc} = \widetilde{u}_{sc} + u_R$ where 
\[
\widetilde{u}_{sc} ( \cdot, \cdot, \tilde\bt ) \in H^{-1}_{\Gamma , loc} \big( (\R^{n} \setminus \overline{B(0,R)}) \times \R \big),  %:=  \{ w \in H^{-1}_{loc}\big( (\R^{n} \setminus \overline{B(0,R)}) \times \R \big) : WF (w) \subset \Gamma \},  
\]
with
\begin{equation}\label{gammaequa}
\Gamma := \bigcup_{j=1}^N N^*\Sigma_j \cup   \{ (x,t \ ; \  \xi, k ) \in T^* \R^{n+1} : (x,t)\in \R^{n+1} ,  \ \langle \xi, \tilde\theta_j  \rangle = 0, \ k=0  \},
\end{equation}
%\comment{What is $\{ \theta_j\}^\perp$?}
%$WF (\widetilde{u}_{sc}) \subset \bigcup_{j=1}^k N^* \Sigma_j $ 
and $u_R \in H^1_{loc} (\R^{n+1})$ for compactly supported $V \in H^2(\R^n) \cap L^\infty (\R^n)$. %Natural extensions of the trace for scattered wave  $\tr^{\pm} ( \widetilde{u}_{sc} , u_R)$,
%\[
%\tr^\pm_{\Sigma_j} : \mathcal{D}'_{\Gamma} \big( (\R^{n} \setminus \overline{B(0,R)}) \times \R \big)  \rightarrow \mathcal{D}' \big(\Sigma_j \cap (\R^{n} \setminus \overline{B(0,R)}) \times \R \big)
%\]
%%
%%Motivated by this, given a decomposition $(\tilde{v}, v_R) \in H^{-1}_{\Gamma,loc} (\R^{n+1})$, we define  %$\widetilde{u}_{sc},u_R$ we define 
Given $\theta_i \in \S^{n-1}$, we are interested in discontinuity across $\Sigma_i$ which appears in the first term of the decomposition. 
Amplitude of the peak is defined by the difference
\begin{equation}\label{def_peak}
\begin{split}
[ u_{sc} ]_{\Sigma_{i}} (x, \tilde\bt) :=&  (\tr^+_{\Sigma_i} - \tr^-_{\Sigma_i})\widetilde{u}_{sc}(x,x\cdot \theta_i, \tilde\bt)
% + \tr_{\Sigma_i } u_R (x, x \cdot \theta_i , \tilde\bt) 
%-  \tr^-_{\Sigma_i} \widetilde{u}_{sc}(x,x\cdot \theta_j, \tilde\bt)
%- \tr_{\Sigma_i}u_R(x, x \cdot \theta_i, \tilde\bt), \\%\text{, }
%= \lim_{\epsilon \rightarrow 0+} \tr_{\Sigma_j} \big(S^*_\epsilon v_1 , v_2   \big) (x,x\cdot \theta_j, \tilde\bt)
%z:= x + s \theta_i .\\
%=& \lim_{\epsilon \rightarrow 0} \big\{ u_{sc} (x,x \cdot \theta_{i} + |\epsilon|, \bt ) -  u_{sc} (x,x \cdot \theta_{i} - |\epsilon|,\bt ) \big\}  \\
%=& \sum_{j=1}^N v_{sc} (x,x \cdot \theta_i,\theta_j) -  \sum_{j=1}^N\lim_{\epsilon \rightarrow 0} v_{sc} (x,x \cdot \theta_i - |\epsilon|, \theta_j ) , \\
%=& \sum_{j=1}^k v_{sc} |_{\Sigma_i} (x,\theta_j) -  \sum_{j=1}^k\lim_{\epsilon \rightarrow 0} v_{sc} (x,x \cdot \theta_i - |\epsilon|, \theta_j )  \\
\end{split}
\end{equation}
between the two limits $\tr^\pm_{\Sigma_i} \widetilde{u}_{sc} :=   \lim_{\epsilon \rightarrow 0 \pm} \tr_{\Sigma_i} \circ  S^*_\epsilon$, 
%\[
%\tr^\pm_{\Sigma_j} \widetilde{u}_{sc} (x,x\cdot \theta_j, \tilde\bt) := \lim_{\epsilon \rightarrow 0 \pm} \tr_{\Sigma_j}  S^*_\epsilon  \widetilde{u}_{sc} (x,x\cdot \theta_j, \tilde\bt) 
%%\lim_{\epsilon  \rightarrow 0\pm}  \Big[ \big( S^*_\epsilon \tilde{u}_{sc}  \big)\Big|_{\Sigma_j} (x,x \cdot \theta_j,\tilde\bt) \Big]+ \tr_{\Sigma_j} (u_R) (x,x \cdot \theta_j,\tilde\bt),
%\]
where 
\[
S^*_\epsilon : H^{-1}_{\Gamma, loc} \big( (\R^{n} \setminus \overline{B(0,R)}) \times \R \big) \rightarrow H^{-1}_{\Gamma_\epsilon, loc} \big( (\R^{n} \setminus \overline{B(0,R)}) \times \R \big), 
\]
\[
\Gamma_\epsilon := \{ (x, t - \epsilon; v ) \in T^* \R^{n+1} : (x,t; v) \in \Gamma \}
\]
 is the pull-back generated by $S_\epsilon : (x,t) \mapsto (x, t + \epsilon)$. %\comment{What is $S_\epsilon$?}
% Amplitude of the leading singularity normal to $\Sigma_i$ is defined by difference between the two traces: 
%\begin{equation}\label{def_peak}
%\begin{split}
%[ u_{sc} ]_{\Sigma_{i}} (x, \tilde\bt) :=& \tr^+_{\Sigma_i} u_{sc} (x, x \cdot \theta_i , \tilde\bt) -  \tr^-_{\Sigma_i}u_{sc}(x, x \cdot \theta_i, \tilde\bt), %\text{, }
%%z:= x + s \theta_i .\\
%%=& \lim_{\epsilon \rightarrow 0} \big\{ u_{sc} (x,x \cdot \theta_{i} + |\epsilon|, \bt ) -  u_{sc} (x,x \cdot \theta_{i} - |\epsilon|,\bt ) \big\}  \\
%%=& \sum_{j=1}^N v_{sc} (x,x \cdot \theta_i,\theta_j) -  \sum_{j=1}^N\lim_{\epsilon \rightarrow 0} v_{sc} (x,x \cdot \theta_i - |\epsilon|, \theta_j ) , \\
%%=& \sum_{j=1}^k v_{sc} |_{\Sigma_i} (x,\theta_j) -  \sum_{j=1}^k\lim_{\epsilon \rightarrow 0} v_{sc} (x,x \cdot \theta_i - |\epsilon|, \theta_j )  \\
%\end{split}
%\end{equation}
As shown in Section \ref{mainprooffi}, the limits exist for compactly supported potentials in $H^2(\R^n) \cap L^\infty (\R^n)$. 
It is a consequence of the trace theorem that the amplitude $[u_{sc}]$ does not depend on the choice of the decomposition
\[
u_{sc} = \widetilde{u}_{sc} + u_R \in H^{-1}_{\Gamma, loc} \big( (\R^{n} \setminus \overline{B(0,R)}) \times \R \big) + H^{1}_{loc} (\R^{n+1}).
\]
%as shown in the lemma below. 
In particular, the formal notation $(\tr^+_{\Sigma_i} - \tr^-_{\Sigma_i}) u_{sc}(x,x\cdot \theta_i, \tilde\bt)$ in the introduction makes sense.

%\begin{lemma}
%Let $\tilde{v}_1,\tilde{v}_2 \in \mathcal{D}_{\Gamma}' \big( (\R^{n} \setminus \overline{B(0,R)}) \times \R \big)$ with $\Gamma$ defined by \eqref{gammaequa} and assume that the limits 
%$
% \lim_{\epsilon \rightarrow 0 \pm} \tr_{\Sigma_i} ( S^*_\epsilon \tilde{v}_j )
% $, $j=1,2$ 
% exist in $\mathcal{D}' \big(\Sigma_i \cap (\R^{n} \setminus \overline{B(0,R)}) \times \R \big)$.  Then, if 
% \[
% \tilde{v}_1 - \tilde{v}_2 \in H^1_{loc}(\R^{n+1}),
% \]
% we have 
% \[
%( \tr^+_{\Sigma_i} - \tr^-_{\Sigma_i}) v_1 = ( \tr^+_{\Sigma_i} - \tr^-_{\Sigma_i}) v_2
% \]
% \begin{proof}
% By continuity of the trace $\tr : H^1_{loc}(\R^{n+1}) \rightarrow H^{1/2}_{loc} (\Sigma_i)$ 
% \[
%  \lim_{\epsilon \rightarrow 0 \pm} \tr_{\Sigma_i} ( S^*_\epsilon \lim_{m\rightarrow \infty} v_m ) = v_m (x, x \cdot \theta_i  - \epsilon \theta_i)
%  \]
% \end{proof}
%\end{lemma}

%\subsection{The Ray Transform}\label{raytransforma}

\subsection{Unique recovery for smooth potentials}\label{proofpro}

Here we prove that arbitrary moments of the random potential are uniquely recovered by the data if we know a priori that the potential is smooth.
This partial result is the basis for the full proof of Theorem \ref{thm:main} in Section \ref{mainprooffi}.

\begin{proposition}\label{pro12378}
Let $V : \mathbb{R}^n \times \Omega \rightarrow \R$ be a random field that such that the condition \textup{(C2)} holds and $V \in C^\infty(\R^n)$ almost surely.
Then $M^k \in \mathcal{E}' \left(\Timk \R^{n}\right)$, given by \eqref{eq:moment_map}, 
is uniquely determined by the exterior data \eqref{eq:data} for any $k\in\mathbb{N}$.
\end{proposition}

\begin{proof}
%We prove the theorem for a potential $V(x,\omega)$ that is compactly supported and smooth function with probability $1$.
By linearity and uniqueness of the solution the scattered wave is of the form 
\[
u_{sc} (x,t,\tilde\bt) = \sum_{j=1}^N v_{sc} (x,t,\tilde\theta_j),
\]
where $v_{sc} (x,t,\theta)$ is the scattered part of $v (x,t,\theta)$, defined by
\begin{eqnarray*}
(\Box -V(x)) v (x,t,\theta)  & = & 0, \\
v (x,t,\theta) & = & \delta ( t -x \cdot \theta ) + v_{sc} (x,t,\theta), \\
v_{sc} (x,t,\theta) & = &  0, \quad {\rm for} \; t = -\text{diam} (K) ,
\end{eqnarray*}
for $(x,t, \theta) \in \R^{n+1}\times \mathbb{S}^{n-1}$.
Let $a_\alpha (x,\theta) \in  C^\infty( \R^n \times \mathbb{S}^{n-1})$, $\alpha \in \mathbb{N} \setminus \{0\}$, be the coefficients of the following asymptotic expansion:
\begin{equation}\label{1245637890hfdsfhjdns}
v_{sc} (x,t, \theta) = \frac{1}{2} a_1 ( x, \theta) H ( t - x \cdot \theta ) + \frac{1}{2} \sum_{\alpha=1}^s  a_{\alpha+1} (x,\theta) ( t - x \cdot \theta  )^\alpha_+  \mod C^{s+2}, 
\end{equation}
where the identity is up to functions in $C^{s+2}(\R^{n+1})$ and $H$ is the Heaviside step function,
\[
H(x) = \begin{cases} 1 ,& \text{for } x\geq 0 \\ 0, & \text{for } x <0. \end{cases}.
\]   
We recall from \cite{MR3224125, petkov1989scattering, 10.2307/2374463} %references!
that for a smooth compactly supported potential the expansion exists and is given recursively by
\[
 \alpha \theta \cdot \nabla a_{\alpha+1} (x,\theta) = \frac{1}{2}  \big( \Delta + V(x) \big) a_\alpha(x,\theta), \ \alpha=1,2,3,\dots
\] 
and
\[
\theta \cdot \nabla a_{1} ( x, \theta) =  V(x) .
\]
Due to the zero initial value of the scattered wave, the first coefficient 
%$a_1 (x,\theta)$ 
equals the line integral: 
%along bicharacteristic curves:
% given in Section \ref{raytransforma}.
%
%Let $a_1 (x,\theta) = a_1(x,\theta,\omega)$, $(x,\theta,\omega)\in \R^n \times \mathbb{S}^{n-1} \times \Omega$ be given by $\theta \cdot \nabla a_1 ( x,\theta,\omega) = V(x, \omega)$ with an initial value $a_1(r \theta + x, \theta, \omega) = 0$, $r \ll 0$. 
%For a continuous potential one derives
%For compactly supported $V \in L^2( \mathbb{R}^n )$ one derives
\[
a_1 (x,\theta) = \int_{-\infty}^0 V(x+s \theta)ds.
\] 
%along bicharacteristic lines.
%considered as a distribution with respect to the spatial variable $x\in \R^n$. 
Therefore, we have the identity
\[
\lim_{s \rightarrow \infty} a_1 (x + s\eta ,\theta)  = \begin{cases} 0, & \eta \neq \theta \\ \int_{\R}  V(x+s \theta)ds, & \eta = \theta .\end{cases}, 
\]
for $\eta \in \mathbb{S}^{n-1}$.
%as a distribution with respect to the variables $ \eta \in \mathbb{S}^{n-1}$ and $x \in \R^n$.
% In particular,
%the boundary data of a single scattering event satisfies 
%\[
%2 \lim_{s\rightarrow \infty} v_{sc} |_\Sigma (x + s \theta , \theta)
%% =  \lim_{s\rightarrow \infty} 2 u_{sc} (x + s \theta , s + x\cdot \theta , \theta)
%   \lim_{s \rightarrow \infty } a_1 ( x + s \theta,\theta) =  \int\limits_{\R} V ( x + s \theta ) ds,
%\]
A restriction of $v_{sc} (x,t, \theta)$ to the boundary $t = x \cdot \theta $ at any distant point along the ray $x + s\theta$, $s \in \R$ is simply the associated ray transform of the potential. 
%We define
%\[
%\begin{split}
%[ u_{sc} ]_{\Sigma_{i}} (x + s \theta_i) :=& u_{sc}|_{\Sigma_{i}+} (z, z \cdot \theta_i) - u_{sc}|_{\Sigma_{i}-} (z, z \cdot \theta_i) \text{, }z:= x + s \theta_i \\
%%=& \lim_{\epsilon \rightarrow 0} \big\{ u_{sc} (x,x \cdot \theta_{i} + |\epsilon|, \bt ) -  u_{sc} (x,x \cdot \theta_{i} - |\epsilon|,\bt ) \big\}  \\
%%=& \sum_{j=1}^N v_{sc} (x,x \cdot \theta_i,\theta_j) -  \sum_{j=1}^N\lim_{\epsilon \rightarrow 0} v_{sc} (x,x \cdot \theta_i - |\epsilon|, \theta_j ) , \\
%%=& \sum_{j=1}^k v_{sc} |_{\Sigma_i} (x,\theta_j) -  \sum_{j=1}^k\lim_{\epsilon \rightarrow 0} v_{sc} (x,x \cdot \theta_i - |\epsilon|, \theta_j )  \\
%\end{split}
%\]
%for 
%as in (\ref{7356444428834}). 
%Here $s$ is assumed to be large enough for $(x,t) \mapsto H ( t-x \cdot  \theta_j )$, $\theta_j\neq \theta_i$ to be constant in some open neighbourhood of the boundary point $(z, z \cdot \theta_i ) \in \R^{n+1}$. 
%We do not define the boundary restrictions for small values of $s$.
%, although it could be done using continuous extensions.
%Assume $\theta_i \neq \theta_j$, for $i\neq j$.
Applying  (\ref{1245637890hfdsfhjdns}) implies
\[
\begin{split}
[ u_{sc} ]_{\Sigma_{i}} (x + s\theta_i , \tilde\btk) 
%&=   \sum_{j=1}^N v_{sc} (x+ s \theta_i,x \cdot \theta_i + s  , \theta_j ) \\
%-& \sum_{j=1}^N\lim_{\epsilon \rightarrow 0} v_{sc} (x+ s \theta_i,x \cdot \theta_i  + s - |\epsilon|, \theta_j )  \\
&=  \frac{1}{2}a_1 (x+ s \theta_i  , \theta_i ) 
\end{split}
\]
%for 
%sufficiently large $s= s(x)>0$ and 
%distinct initial waves, i.e. $\theta_i \neq \theta_j$, for $i\neq j$.
Therefore, the data at points $\bxk \in \Timk \R^{n}$ in directions $\btk \in \Timk \mathbb{S}^{n-1}$ satisfies
\begin{equation}\label{54206937632874628374}
\begin{split}
D^k( \bxk , \btk)  
&=   \mathbb{E} \Big( \lim_{s\rightarrow \infty} \prod_{j=1}^k [u_{sc}]_{\Sigma_1} ( x_j + s \theta_j , \tilde\bt) \Big) \\
%&= \frac{N}{2^k} \mathbb{E} \big(  a_1 (x_1+ s \theta_1,x \cdot \theta_1 + s  , \theta_i )\cdots a_1 (x_k+ s \theta_k,x_k \cdot \theta_k + s  , \theta_k ) \big) \\
&= \frac{1}{2^k} \mathbb{E} \Big(\prod_{j=1}^k \int\limits_{\R}  V ( x_j + s_j \theta_j ) ds_j   \Big) \\
&=  \frac{1}{2^k}\int\limits_{\R^k } \mathbb{E} \Big(  \prod_{j=1}^k V ( x_j + s_j \theta_j )    \Big) ds_1 \dots ds_k \\
&= \frac{1}{2^k} \int\limits_{L (\bxk,\btk)} \mathbb{E} \Big(\prod_{j=1}^k V ( z_j)   \Big) dl(z_1, \dots ,z_k), \\
\end{split}
\end{equation}
where 
%$N:=N( \btk) := \Pi_{i=1}^k \sum_{j=1}^k \delta( |\theta_i- \theta_j| ) $, 
$L(\bxk,\btk)$ is the affine subspace 
%\[L(x,\theta):=L(\bxk, \btk) \subset \Timk\R^n,\] is a shifted plane, 
\begin{equation}
	\label{eq:defL}
L(\bxk, \btk) := \{  \bxk +\boldsymbol{h} \in \prod_{j=1}^k \R^n : \boldsymbol{h}\in  H_{\btk} \},
\end{equation}
where
\begin{equation*}
H_{\btk} := \{ (s_1\theta_1,\dots,s_k \theta_k) \in \Timk\R^n  : s_1,\dots,s_k \in \R \},
\end{equation*}
and $dl(z_1,\dots,z_k)$ denotes the pull-back volume form which is induced from the canonical volume form via the inclusion map $L(\bxk,\btk) \hookrightarrow \Timk\R^n$. 

The main idea of the next step is as follows:
Based on the identity (\ref{54206937632874628374}) and the fact that every $(nk-1)$-dimensional shifted hyperplane in $\Timk\R^n$ is obtained by stacking up spaces $L(\bxk, \btk)$ with different parameters $\bxk$, $\btk$ we can reconstruct the Radon transform of the $k^{\text{th}}$ moment map, $M^k(\bxk) := \mathbb{E} \Big( \prod_{j=1}^k V (x_j)\Big)$ from the data. 
That is to say, for arbitrary $(r, \boldsymbol\eta)  \in\R \times \mathbb{S}^{nk-1}$ 
one divides the hyperplane
\begin{equation}
	\label{eq:hyperplane_gamma}
	\Gamma ( r,\boldsymbol{\eta}) := \left\{ \boldsymbol{z} \in\Timk\R^n :  \boldsymbol{z}\cdot \boldsymbol{\eta} = r \right\}
\end{equation}
%\[
%\Gamma_{r,\eta} := \{ z \in \Timk\R^n :  \eta \cdot z =r \},
%\] 
into distinct subspaces of the form $L(\bxk, \btk)$ given in equation \eqref{eq:defL}, then applies the data together with (\ref{54206937632874628374}) to obtain integrals $\int_L M^k  dl$ of the moment map over each of the subspaces and finally computes the superposition of them.

Let us formulate this idea rigorously. First, we provide a representation of the hyperplane $\Gamma(r,\boldsymbol\eta)$
as an orthogonal decomposition involving the subspace $L$ in \eqref{eq:defL}. 
Construct arbitrary smooth functions  
\[
\theta_j  : \mathbb{S}^{nk-1}  \rightarrow  \mathbb{S}^{n-1} %, \ j=1,\dots,k, %\ \theta_j = \theta_j(\boldsymbol\eta)
\]
to satisfy $$\eta_j \cdot \theta_j ( \boldsymbol\eta) = 0$$
for any $j=1,...,k$.
Next, let $\btk := ( \theta_1 , \dots, \theta_k)$ and define the set $$P ( \boldsymbol{\eta} , \btk) \subset \Timk  \R^{n}$$ to be the maximal linear subspace orthogonal to $\boldsymbol{\eta}$ and vectors
\[
(0,\dots, 0, \underbrace{\theta_j ( \boldsymbol\eta )}_{j^{\text{th}} \text{ slot}},0,\dots,0) \in
% \R^n \times \cdots \times  \R^n = 
\Timk\R^n, \ j=1,\dots,k,
\]
simultaneously.
%We shall omit the variable $\boldsymbol\eta$ from $\theta_j ( \boldsymbol\eta)$. 
% for every $j=1,\dots,k$. 
Clearly, each $\bxk \in \Gamma ( r,\boldsymbol{\eta})$  is uniquely written in form
$
%\[
\bxk =  
%x^\theta_j \widehat\theta_j 
\bxk^L
 + \bxk^P,
%\]
$
 where 
\[
\bxk^L =(x^L_1  , \dots, x^L_k )  \in L( r\boldsymbol{\eta},\btk):= L(r\eta_1,\dots,r\eta_k,\btk),
\]
$x^L_j \in \R^n$ for every $j=1,\dots,k$, and
\[
 \bxk^P = ( x^P_1, \dots,x_k^P)  \in P ( \boldsymbol{\eta} , \btk),
 \] 
 where $x^P_j \in \R^n$ for $j=1,\dots,k$. In consequence, we have that
\[
\Gamma ( r,\boldsymbol{\eta}) = L( r\boldsymbol{\eta},\btk) \oplus P ( \boldsymbol{\eta} , \btk).
%r \eta \oplus \R \widehat\theta_1 \oplus \cdots \oplus \R \widehat\theta_k \oplus P ( \eta , \btk)
\]
Now, by identity (\ref{54206937632874628374}), the Radon transform of the moment function $M^k$ at $(r, \boldsymbol{\eta})  \in \R \times \mathbb{S}^{nk-1}$ takes the form
 \begin{eqnarray}
 \label{532684892}
R [M^k] (r,\boldsymbol{\eta}) & = & \int\limits_{\Gamma(r,\boldsymbol{\eta})} M^k (\bxk) d\nu (\bxk)\nonumber \\
&=&  \int\limits_{P ( \boldsymbol{\eta} , \btk)} \int\limits_{L ( r\boldsymbol{\eta}  , \btk)}\mathbb{E} \Big( \prod_{j=1}^k V (x_j^L+ x_j^P) \Big) dl (\bxk^L)dP(\bxk^P)\nonumber  \\
&=&  \int\limits_{P ( \boldsymbol{\eta} , \btk)} \int\limits_{L ( r\boldsymbol{\eta} + \bxk^P , \btk)}\mathbb{E} \Big( \prod_{j=1}^kV (z_j)\Big) dl (\boldsymbol{z})dP(\bxk^P) \nonumber \\
&=& 2^k \int\limits_{P ( \boldsymbol{\eta} , \btk)} D^k(r\boldsymbol\eta + \bxk^P, \btk (\boldsymbol\eta) )   dP ( \bxk^P ), 
%&= 2k \Big\langle \delta_{P ( \boldsymbol{\eta} , \btk)}   \delta \big(\bxk \cdot \btk( \boldsymbol\eta) \big) ,  D^k ( \bxk , \btk ( \eta ) \big)\Big\rangle
 \end{eqnarray}
%\begin{equation}\label{532684892}
%\begin{split}
%&R [M^k] (r,\boldsymbol{\eta}) = \int\limits_{\Gamma(r,\boldsymbol{\eta})} M^k (\bxk) d\nu (\bxk) \\
%&= \int\limits_{P ( \boldsymbol{\eta} , \btk)} \int\limits_{L ( r\boldsymbol{\eta}  , \btk)}\mathbb{E} \Big( \prod_{j=1}^k V (x_j^L+ x_j^P) \Big) dl (\bxk^L)dP(\bxk^P)  \\
%&= \int\limits_{P ( \boldsymbol{\eta} , \btk)} \int\limits_{L ( r\boldsymbol{\eta} + \bxk^P , \btk)}\mathbb{E} \Big( \prod_{j=1}^kV (z_j)\Big) dl (\boldsymbol{z})dP(\bxk^P)  \\
%&= 2^k \int\limits_{P ( \boldsymbol{\eta} , \btk)} D^k(r\boldsymbol\eta + \bxk^P, \btk (\boldsymbol\eta) )   dP ( \bxk^P ) \\
%%&= 2k \Big\langle \delta_{P ( \boldsymbol{\eta} , \btk)}   \delta \big(\bxk \cdot \btk( \boldsymbol\eta) \big) ,  D^k ( \bxk , \btk ( \eta ) \big)\Big\rangle
%\end{split}
%\end{equation}
where  $dP$, $dl$, and $d\nu$ are the canonical volume forms, induced by inclusions into $ \Timk\R^n$.
Let $\iota_{\btk(\boldsymbol{\eta} ) } :  T_{\btk(\boldsymbol\eta) }  \Timk \mathbb{S}^{n-1} \hookrightarrow T  \Timk \mathbb{S}^{n-1} $ be the trivial inclusion. 
Considering the exterior data as a function $ D^k  \in  C^\infty_c \big( T \Timk  \mathbb{S}^{n-1} \big)$ reduces (\ref{532684892})  to 
\begin{equation}\label{nice_form}
R [M^k] (r,\boldsymbol{\eta}) = \big\langle   \delta \big( r -\boldsymbol\eta \cdot \iota_{\btk (\boldsymbol\eta)} ( \boldsymbol{v} ) \big) ,  D^k \circ \iota_{\btk(\boldsymbol{\eta}) }  (\boldsymbol{v} )   \big\rangle  = \big\langle ( \iota_{ \btk (\boldsymbol\eta)} )_*   \delta \big( r -\boldsymbol\eta \cdot \iota_{\btk (\boldsymbol\eta)}  \big) ,  D^k   \big\rangle  ,  %= \btk_* R[   D^k](r,\boldsymbol{\eta})
\end{equation}
%where $ \eta \cdot (\btk (\boldsymbol\eta), \boldsymbol{v}) \in N  \mathbb{S}^{nk-1} $
%where $\btk^*  :  C^\infty \big( T  \Timk  \mathbb{S}^{n-1} \big) \rightarrow C^\infty \big(  T \mathbb{S}^{nk-1} \big)$.
where we denote $ \boldsymbol{v}\in  T_{\btk(\boldsymbol\eta) }  \Timk  \mathbb{S}^{n-1} $ 
%$\langle \iota_* u , \varphi \rangle:=  \langle  u , \iota^* \varphi \rangle$ 
and identify $ T_\theta \mathbb{S}^{n-1} = \{ \theta \}^\perp \times \{ \theta \}  \subset \R^n \times \R^n $.
%Let 
%%\[
%%A^k :  \mathbb{S}^{nk-1} \rightarrow  \mathcal{D' } (T  \Timk  \mathbb{S}^{n-1} ) ,
% $(r, \eta) \mapsto A^k ( r, \boldsymbol\eta ) \in \mathcal{D' } \Big( T_{\btk( \boldsymbol\eta ) } \Timk \mathbb{S}^{n-1} \Big)$ be defined by kernel $ \delta \big( r -\boldsymbol\eta \cdot (\btk (\boldsymbol\eta) ; \boldsymbol{v}) \big)$ above.
%\]
% be defined by 
%\[
%\xymatrix{
%D^k &  D^k \circ \btk \\
%c&R [M^k]
%}
%\]
Finally, to obtain $M^k$ from the transformed quantity one applies the Radon inversion formula, 
\[
\Delta^{(nk-1)/2} R^* R = id,
\] 
where $R^*$ is the adjoint of the Radon transform and 
\[
 \Delta^{(nk-1)/2} f (x) := \int_{\R^{nk}} e^{ix \cdot \xi} | \xi |^{nk-1} \widehat{f} (\xi) d\xi .
 \]
 More precisely, by (\ref{nice_form}), the Radon inversion formula implies
 \[
 \Psi^k \{   D^k \} (\bxk)  = M^k (\bxk)
 \]
where
 \begin{equation}\label{ReOp}
 \Psi^k : = \Delta^{(n-1)/2} R^t  A^k ,
\end{equation}
and
\[
A^k : 
%C^\infty ( \Timk\R^n \times \mathbb{S}^{k(n-1)} ) 
%\color{red}(*)\color{black} 
C^\infty_c \left( \Timk T \mathbb{S}^{n-1}\right) \rightarrow C^\infty (  \R \times \mathbb{S}^{kn-1} ), \ A^k( r, \boldsymbol\eta   ) := ( \iota_{ \btk (\boldsymbol\eta)} )_*   \delta \big( r -\boldsymbol\eta \cdot \iota_{\btk (\boldsymbol\eta)}  \big).
\]
This concludes the proof.
\end{proof}
%is given by the kernel $\btk (\boldsymbol\eta)_*  \delta \big( r -\boldsymbol\eta \cdot (\btk (\boldsymbol\eta) ; \boldsymbol{v}) \big)$.
%\[
%A^k f (r,\boldsymbol\eta) = 2^k \int\limits_{P ( \boldsymbol{\eta} , \btk ( \boldsymbol\eta ))} f(  r \boldsymbol\eta + \bxk^P, \btk( \boldsymbol\eta ) )   dP ( x^P ),
%\]
We will refer to the map $ \Psi^k$ 
%which transforms exterior data of a smooth compactly supported potential to the corresponding moment map 
as a reconstruction operator for smooth potentials.

\subsection{Proof of Theorem \ref{thm:main}, The General Setting}\label{mainprooffi}

We shall first study the system for a single random parameter $\omega_0 \in \Omega$.
Fix the potential $V(x)=V(x,\omega_0) \in L^\infty ( \R^n) \cap H^2 (\R^n )$ that is supported in $B(0,R)$, $R >0$ and 
write the solution of (\ref{eq:main_problem}) with a possibly non-smooth potential in the form
\begin{equation}\label{2504778798787878}
u_{sc} (x,t,\tilde\bt) =\sum_{j=1}^N  \frac{1}{2} a_1(x, \tilde\theta_j)  H (t- x\cdot \tilde\theta_j) + u_R (x,t,\tilde\bt),
\end{equation} 
where
\[
 a_1 (x, \tilde\theta_j ) := \int\limits_{-\infty}^0 V (x + s\tilde\theta_j ) ds 
\]
and $u_R (x,t,\tilde\bt)$ is the residual term.
We will need the following lemmas:
\begin{lemma}
Let $V(x) \in   L^\infty (\R^n) \cap H^2(\R^n)$ be supported in $B(0,R)$. It holds that $x \mapsto H(t-x\cdot \theta) \Delta a_1(x,\theta) \in L^2 ( \R^n)$.
\begin{proof}
Consider $\Delta a_1(x,\theta)$ as a linear functional
\[
\Delta a_1(\cdot,\theta) : L^2( \R^n)  \rightarrow \R : \phi(x) \mapsto  \int\limits_{-\infty}^0  \int\limits_{\R^n}  \Delta V(x + s\theta) \phi (x)  dx ds.  \\
%=   \int\limits_{-\infty}^0 \Delta \tilde{V} * \phi (-s \theta)  ds,
\]
%where $\tilde{V} (x) := V(-x)$.
It follows by the Cauchy--Schwarz inequality that
\begin{eqnarray*}
\left| \langle H(t-x \cdot \theta ) \Delta a_1(x,\theta) , \phi(x)  \rangle \right|  & = & \left| \int\limits_{-\infty}^0  \int\limits_{\R^n}  \Delta V(x + s\theta) \phi (x) H(t-x\cdot \theta)  dx ds \right|\\
&\leq & 
%\int\limits_{\R^n}  \int\limits_{-t-2R}^0  | \Delta V(x + s\theta) \phi (x) | ds dx 
%=
  \int\limits_{-t-2R}^0 \int\limits_{\R^n}   | \Delta V(x + s\theta)| |\phi (x) | dx ds \\
& \leq &\int\limits_{-t-2R}^0    \| \Delta V (\cdot + s\theta) \|_{L^2( \R^n)} \|\phi \|_{L^2( \R^n)}  ds  \\
 &=& (2R + t)  \| \Delta V\|_{L^2( \R^n)} \|\phi \|_{L^2( \R^n)} ,
%\int\limits_{\R^n}  \int\limits_{-t - 2R}^0 |\Delta V(x + s\theta)|| \phi (x) | ds dx\\
%&\leq   \int\limits_{-t - 2R}^0\int\limits_{\R^n} |\Delta V(x + s\theta)|| \phi (x) | dx ds 
\end{eqnarray*} 
%\[
%\begin{split}
%&| \langle H(t-x \cdot \theta ) \Delta a_1(x,\theta) , \phi(x)  \rangle |  = \Big| \int\limits_{-\infty}^0  \int\limits_{\R^n}  \Delta V(x + s\theta) \phi (x) H(t-x\cdot \theta)  dx ds \Big|\\
%&\leq  
%%\int\limits_{\R^n}  \int\limits_{-t-2R}^0  | \Delta V(x + s\theta) \phi (x) | ds dx 
%%=
%  \int\limits_{-t-2R}^0 \int\limits_{\R^n}   | \Delta V(x + s\theta)| |\phi (x) | dx ds 
% \leq \int\limits_{-t-2R}^0    \| \Delta V (\cdot + s\theta) \|_{L^2( \R^n)} \|\phi \|_{L^2( \R^n)}  ds  \\
% &= (2R + t)  \| \Delta V\|_{L^2( \R^n)} \|\phi \|_{L^2( \R^n)} ,
%%\int\limits_{\R^n}  \int\limits_{-t - 2R}^0 |\Delta V(x + s\theta)|| \phi (x) | ds dx\\
%%&\leq   \int\limits_{-t - 2R}^0\int\limits_{\R^n} |\Delta V(x + s\theta)|| \phi (x) | dx ds 
%\end{split}
%\]
that is, the functional $H(t-x\cdot \theta) \Delta a_1(x,\theta) : L^2(\R^n) \to \R$ is bounded.  The claim follows by duality.
%identifying $L^2 ( \R^n)$ with the dual $(L^2( \R^n))^*$.

\end{proof}
\end{lemma}

\begin{lemma}
The residual term $u_R$ in \eqref{2504778798787878} satisfies $u_R (\cdot,t,\tilde\bt) \in H^1_{loc} (\R^{n+1} )$. 
%Let us assume so far th
\begin{proof} 
Let $T>R$. 
Substituting the ansatz (\ref{2504778798787878}) into (\ref{eq:main_problem}) yields
\begin{equation}\label{EQU2233}
\begin{split}
(\Box - V(x)) u_R (x,t,\tilde\bt) = \frac{1}{2} \sum_{j=1}^N H (t- x\cdot \tilde\theta_j) ( V(x) + \Delta ) a_1(x, \tilde\theta_j) .
%\\  \in L^1 \big( [ -T ,T ] ; L^2 (  B(0, T +2R  ) ) \big),
\end{split}
\end{equation}
The idea is to construct a sufficiently regular solution of the equation 
 (\ref{EQU2233})  
by extending a local solution that satisfy initial and boundary conditions of the residual. 
The right side of $(\ref{EQU2233})$ is a sum of $L^2$-functions supported in $B(0,R+t)$ for each $t$. In particular, 
\begin{equation}\label{1069360389089}
 (\Box - V(x)) w_R (x,t,\tilde\bt) \in L^1 \big( [ -R , T ] ; L^2 (  B(0, r  ) ) \big) 
\end{equation}
where $w_R $ refers to a solution of (\ref{EQU2233}) in a smaller set $(x,t) \in B(0, r  ) \times [-R , T ] $, with sufficiently large $r>2R+T$.
%Let us assume so far that solution $u_R^T \in H^1( B(0, 2T +2R  ) \times [-T,T] )$.
%for every $T> R$. 
%By \cite[Ch. IV \textsection3]{lohwater2013boundary}, 
%the equation (\ref{1069360389089}) together with the initial and boundary conditions,  
%We shall derive necessary initial and boundary conditions for the solution. 
%We are interested in a family of local solutions $u_R^T$, $T>R$ that satisfy similar initial and boundary conditions as $u_R $. 
%The idea is to construct a sufficiently regular residual from such local solutions that satisfy initial and boundary conditions of the residual.  
As no scattering occurs before the incident wave hits the potential, we consider the initial conditions
\begin{equation}\label{raja111}
\begin{cases}
w_R (x,t,\tilde\bt )=u_R(x,t,\tilde\bt ) = 0, & t = -R,\\
 \partial_t w_R(x,t, \tilde\bt )= \partial_t u_R(x,t, \tilde\bt )  = 0 , &t =-R.\\
%u_R^T (x,t, \tilde\bt) = 0, & (x,t) \in \partial  B(0, 2T' +2R  ) \times [-T',T'] , \text{ for } \ T' \in (R,T]
\end{cases}
\end{equation}
%Within the space $\R^{n+1} \setminus \overline{B(0,R)}$ the residual  
%
In $ \big(\R^n \setminus \overline{B(0,R)} \big)\times \R$ the scattered wave satisfies the free wave equation $\square u_{sc} = 0$ so due to the initial values $u_{sc} (x,t,\tilde\bt) = 0 = \partial_t u_{sc}  (x,t,\tilde\bt)$, $t<-R$, and unit propagation speed of disturbances the scattered wave is supported within the set  
\[
C_R:= \{ (x,t) \in \R^{n+1} :  x \in B(0, 2R + t) \}
\] 
which contains also the support of $a(x,\theta) H(t-x\cdot \theta)$ for each $\theta \in \mathbb{S}^{n-1}$. Consequently, $\text{supp} (u_R)  \subset C_R$. It is therefore suitable to associate $w_R$ with the zero boundary conditions,
\[
w_R (x,t, \tilde\bt)  = 0 \text{, for } (x,t) \in \partial B(0,r ) \times [-R,T] 
\]
%
%
%
%Within each space $\big( \R^{n+1} \setminus B(0, T' +R  ) \Big) \times [-\infty ,T']$, $T'\in (R,T]$ the distribution $u_R^T $ has to satisfy the free wave equation $\square u_R^T = 0$ so %due to unit  propagation speed of disturbances and (\ref{raja111}) one obtains 
%\begin{equation}\label{raja222}
%u_R^T (x,t, \tilde\bt) = 0 \ \text{for } (x,t) \in \big(B(0, 2T + 2R )\setminus B(0,2T' + 2R) \big) \times [-T',T'] , \ T'\in (R,T]
%\end{equation}
By \cite[Ch. IV \textsection3]{lohwater2013boundary} the initial and boundary conditions above ensure existence and uniqueness of $w_R (x,t,\tilde\bt) \in H^1( B(0, r ) \times [-R,T] )$.
Within the space $\big( B(0,r ) \setminus B(0, T +R  ) \Big) \times [-R ,T]$ the function $w_R $ satisfies the free wave equation, $\square w_R = 0$, so applying the energy estimate together with
 zero initial 
and boundary 
 conditions one extends the domain of $w_R$ into $\R^{n+1}$. 
This can be seen by first applying the estimate to find a zero extension $\overline{w}_R \in H^1 (\R^n \times [-R,T] )$ of $w_R$ with respect to variable $x$ and, second, to extend the domain to the whole timeline $t\in \R$ in the weak sense. 
%\begin{equation}\label{raja222}
%u_R^m (x,t, \tilde\bt) = 0 \ \text{for }  2R(m'+1)  < |x|< r_m - R(1 + m'), \ \ t \in [-R,Rm']
%%\big(B(0, r_m )\setminus B(0,2T + 2R) \big) \times [-R,T'] , \ T'\in (R,T].
%\end{equation}
%In particular, $u_R^T$ satisfies the boundary conditions of each $u_R^{T'}$, $T' \in (R,T]$. 
%implies uniqueness and existence of the solution $u_R^T (x,t,\tilde\bt) \in H^1( B(0, 2T +2R  ) \times [-T,T] )$ %of equation (\ref{1069360389089})
% By uniqueness of $u_R^{T'}$, we conclude that 
% making $T>R$ in (\ref{1069360389089}) smaller affects to $u_R^T$ only by restriction, that is, 
% \[
% u_R^{T'} = u_R^T \big|_{[-R,T'] \times B(0,2T' +3R)}% \in H^1( B(0, 2T' +3R  ) \times [-R,T'] )
% \]
%  for every $T'\in (R, T]$. %That is, 
%Therefore, the weak limit, 
%\[
%\langle \tilde{u}_R (x,t,\tilde\bt) , \phi(x,t) \rangle  := \lim_{T \rightarrow. \infty} \langle u_R^{T} (x,t,\tilde\bt) , \phi(x,t) \rangle  , \ \phi \in C_c^\infty (\R^{n+1} ),  
%\]
% %by arbitrariness of $T>R$,
%is a well defined solution of $(\ref{EQU2233})$ in $(x,t) \in \R^{n+1}$, and  
%$
%\tilde{u}_R (x,t,\tilde\bt)
%% := \lim_{T \rightarrow \infty} u_R^T (x,t, \bt) 
%\in H^1_{loc}( \R^{n+1}).  
%$
%%In particular,  the trace map to define a restriction of $u_R (x,t,\bt)$ to each of the subspaces 
Finally, uniqueness of $u_{sc}$ implies that the extension equals $u_R$. 
\end{proof}
\end{lemma}
Due to regularity of $u_R$ the peak 
% the restrictions $\tr^+_{\Sigma_i} u_R$, $\tr^-_{\Sigma_i} u_R$ to the hyperplanes 
%$
% t= x\cdot \theta_i 
%$, 
%$i=1,\dots,k$ are identical,   
%thus implying that the difference   
%the restriction is defined by using the trace operator, 
%\[
%[ u_R  ]_{\Sigma_i} := u_R ( \cdot, \cdot, \bt )|_{\Sigma_i +} - u_R ( \cdot, \cdot, \bt )|_{\Sigma_i -} ,
%\] 
%between the trace maps $\cdot|_{\Sigma_i \pm} $, 
$[ u_R  ]_{\Sigma_i}$ vanishes.
%associated to the half-planes $0 < \pm (t-x \cdot \theta_i)$, 
%vanishes. % for almost every $(x,x \cdot \theta_i )\in \Sigma_i$.
As in the smooth case, the amplitude of the peak singularity carried by the scattered wave reduces to $ \frac{1}{2} a_1 ( \cdot,\theta_i)$ on distant regions within the boundary $\Sigma_i$:
 \[
\begin{split}
[ u_{sc}]_{\Sigma_i} (x+ s \theta_i , \tilde\bt) %&=[ u_{sc} (\cdot, \cdot, \tilde\bt) ]_{\Sigma_i} (z, z \cdot \theta_i ) \\
&= \left[ \sum_{j=1}^N \frac{1}{2}a_1( x,\tilde\theta_j) H (t- x \cdot \tilde\theta_j )  \right]_{\Sigma_i} (x + s \theta_i, \tilde\bt) %+ [u_R]_{\Sigma_i} (z, z \cdot \theta_i, \tilde\bt)
\\
 %&=  \sum_{j=1}^N \frac{1}{2}a_1( x,\theta_j) H \big(x \cdot (\theta_i -   \theta_j ) \big) 
%- \lim_{\epsilon \rightarrow 0} \sum_{j=1}^N \frac{1}{2} a_1 (x, \theta_j  ) H \big(x \cdot (\theta_i -   \theta_j ) - | \epsilon | \big)    \\
%+ [ u_R(x,\theta_1, \dots, \theta_k) ]_{\Sigma_i} \\
%&= \Big[ \sum_{j=1}^N \frac{1}{2}a_1( x,\theta_j) H \big(t- x \cdot  \theta_i \big) \Big]_{\Sigma_i} \\
&=  \frac{1}{2}a_1(x+ s \theta_i ,\theta_i), %\ \ z :=x+ s \theta_i
\end{split}
\]
%where  the equation is considered for almost every $x\in \R^n$, $\theta_j \neq \theta_i$ for $j\neq i$, $z :=x+ s \theta_i$, and $s$ is chosen large enough for $(x,t)\mapsto H(t-x\cdot \theta_j)$, $\theta_j \neq \theta_i$ to be constant in some neighbourhood of $(z, z \cdot \theta_i)$.
%where 
%\[
%\chi_{i,j} (x) := \begin{cases}1, &\text{if }x \cdot \theta_i  = x \cdot  \theta_j ,  \\ 0, & \text{if }   x \cdot \theta_i \neq x \cdot  \theta_j. \end{cases}
%\]
%Here $T>R$ is chosen large enough for $(x, x \cdot \theta)$ to be contained in $\Sigma^T$.
Consequently, we have
\[
 \lim_{s\rightarrow \infty} [ u_{sc}  ]_{\Sigma_i} (x + s \theta_i ) =  \frac{1}{2}  \int\limits_{\R} V(x + s \theta_i ) ds .
\]
%for distinct initial plane waves.

We shall now apply the previous observations to the probabilistic setting.
As earlier, define the data for $k$ base points by
\begin{equation}\label{2603793639874}
D^k(\bxk, \btk ) 
:=  \mathbb{E} \bigg( \lim_{s\rightarrow \infty} \Timk [u_{sc}]_{\Sigma_j} (x_j + s\theta_j , \tilde\bt) \bigg). 
%&= \mathbb{E} \big(  a_1(x_1 + L(x_1,\theta_1,\omega), \theta_1, \omega)\cdots a_1(x_k +  L(x_k,\theta_k,\omega), \theta_k, \omega) \big),\\
%\text{ a.e. } (x_1, \dots, x_k)\in \Timk\R^n
\end{equation}
%which can also be understood as a generalized function $\bt \mapsto D^k(\bt) \in \mathcal{D}' ( T_{\bt} \Timk \mathbb{S}^{n-1} )$ due to symmetry (\ref{shift}). 
%Again, in the definition of the exterior data we do not require the directions $\btk$ to be distinct. 
%The identities in (\ref{2603793639874}) hold only for almost every $(\bxk)$, that is, outside a set of measure zero. Let $\phi(x) \in C_c^\infty ( \R^n )$.
For $k=1$, taking convolution with the mollifier $\phi_\epsilon (x) = \epsilon^{-n} \phi (  \epsilon^{-1} x)$, $\phi  \in C_c^\infty ( \R^n ) $, $\int_{\R^n}\phi (x) dx = 1$ yields 
\begin{eqnarray}
[ \phi_\epsilon  * D^1( \cdot, \theta) ](x)&= & \int\limits_{\R^n} \phi_\epsilon (z)  \mathbb{E}\left( \frac{1}{2}  \int\limits_\R V(x-z + s\theta) ds \right) dz\nonumber \\
&=& \mathbb{E} \left(\frac{1}{2}  \int\limits_\R V_\epsilon(x+s\theta) ds\right) \nonumber \\
&=&  D^1_{\epsilon} (x,\theta), \label{kase1}
\end{eqnarray}
where $D^1_\epsilon (x,\theta)$  is the exterior data associated to the mollified smooth potential $V_\epsilon (x) := (\phi_\epsilon * V ) (x)$.
%and $\tau_a f(x) = f(x+a)$. 
%In similar fashion, 
Similarly, one derives
\[
[  \phi_\epsilon^k* D^k (  \cdot ,\btk ) ] (\bxk) = D^k_\epsilon (\bxk ,\btk ) ,
\]
where $\phi_\epsilon^k (z_1,\dots,z_k) := \phi_\epsilon (z_1)\cdots \phi_\epsilon(z_k)$, and  $D_\epsilon ( \bxk ,\btk) $ is the data (\ref{54206937632874628374}) at $\bxk \in \Timk \R^n$ in directions $\btk \in \Timk  \mathbb{S}^{n-1}$ for the smooth potential $V_\epsilon$. 
The mollified moment $M^k_\epsilon$ approximates $M^k$ in the sense of generalized functions:
\begin{equation}\label{sdfjkhsd234267384}
\begin{split}
%\lim_{\epsilon \rightarrow 0}  \Delta^{(n-1)/2} R^t \int\limits_P  M^k_\epsilon *  D   dP=\lim_{\epsilon \rightarrow 0}  \Delta^{(n-1)/2} R^t \int\limits_P  D_\epsilon   dP =  
\lim_{\epsilon \rightarrow 0} M^k_\epsilon (\bxk) &= \lim_{\epsilon \rightarrow 0} \mathbb{E} \int_{\Timk\R^n} \phi_\epsilon^k (x_1 - z_1,\dots,x_k-z_k) V (z_1) \cdots V(z_k)  dz
\\
 &= \lim_{\epsilon \rightarrow 0} \int_{\Timk\R^n} \phi_\epsilon^k (x_1 - z_1,\dots,x_k-z_k) \mathbb{E} (V (z_1) \cdots V(z_k))  dz \\
 &=\lim_{\epsilon \rightarrow 0} \phi^k_\epsilon * M^k(\bxk) \\
  &= \delta_0 * M^k(\bxk) \\
 &= M^k (\bxk) \in \mathcal{D}' \left( \Timk\R^n \right),
\end{split}
\end{equation}
see e.g. \cite[Ch 5]{friedlander1998introduction}.  
%and further the reconstruction strategy
%\[
%\xymatrix{
%D^k \ar[r]^{\phi_\epsilon^k*} &D^\epsilon\ar[r]^{\Psi^k}&  M^k_\epsilon \ar[r]^{\lim}  &M^k
%}
%\]
%where the arrow in the middle is the reconstruction map for smooth potentials which, according to (\ref{532684892}) and the Radon inversion formula, is well defined by
%%\begin{equation}\label{6244545454545421828}
%%\Psi^k : \color{red}(*)\color{black}\mathcal{S} ( \Timk\R^n \times \mathbb{S}^{k(n-1)} ) \rightarrow C^\infty ( \Timk\R^n ) : 
%%\end{equation}
%\[
% \Psi^k : = \Delta^{(n-1)/2} R^t  A^k,
%\]
%%
%%
%\[
%A^k : 
%%C^\infty ( \Timk\R^n \times \mathbb{S}^{k(n-1)} ) 
%%\color{red}(*)\color{black} 
%C^\infty_c ( \Timk\R^n \times \mathbb{S}^{k(n-1)} ) \rightarrow C^\infty (  \R \times \mathbb{S}^{kn-1} ) 
%\]
%\[
%A^k f (r,\eta) = 2^k \int\limits_{r\eta + P ( \eta , \theta)} f(x^P , \theta_1 ( \eta) ,\dots,\theta_k ( \eta) )   dP ( x^P ),
%\]
%where the parameters $\theta_1( \eta),\dots,\theta_k ( \eta) \in \mathbb{S}^{n-1}$ are smoothly related to $\eta$ such that $\theta_j \cdot \eta_j = 0$, $j=1,\dots,k$. 
%Here $\eta = (\eta_1,\dots, \eta_k)$ is considered as a unit vector in $\Timk \R^n$.
%%We will refer to the map 
%%as the reconstruction map for smooth potentials. 
The limit in (\ref{sdfjkhsd234267384}) is to be understood by means of sequential convergence in the topology of distributions. 

In summary, we obtain a reconstruction strategy which consists of the following three steps:
\begin{enumerate}
\item[(Step 1)] \textbf{Regularisation of the data: }
The operator 
\[
%\begin{split}
\Phi^k   : C^\infty \bigg( \Timk \mathbb{S}^{n-1}  ;   \mathcal{D}'  \Big( \Timk \R^n \Big) \bigg) \rightarrow C_c^\infty  \Big( (0,\infty) \times T \Timk \mathbb{S}^{n-1}  \Big), 
%\\
\]
given by
\[ 
 \Phi^k( v  ) (\epsilon,  \boldsymbol{y}) := [ \phi_\epsilon * v ( \cdot, \bt) ] ( \boldsymbol{y}), \text{ for } \epsilon \in (0,\infty), \  \boldsymbol{y} \in T_{\bt} \Timk \mathbb{S}^{n-1},% \subset \Timk \R^n \times \R^n % = \Timk \{\theta_j \}^\perp \times \{ \theta_j \}
%\end{split}
\]
transforms the exterior data into a parametrised family of data which      %, $\{D_\epsilon\}_{\epsilon \in (0,\infty)}$, 
corresponds to the regularised potentials:
%$\{V_\epsilon \}_{\epsilon \in (0,\infty)}$ 
\[
 \Phi^k ( D^k )   (\epsilon, \boldsymbol{y}) =  D_\epsilon^k( \boldsymbol{y}, \btk ) , \text{ for }  \boldsymbol{y} \in T_{\bt} \Timk \mathbb{S}^{n-1}.
 \] 
 Above, the regularised data is identified with a smooth function on $T \Timk \mathbb{S}^{n-1}$ according to the invariance (\ref{shift}). 
\item[(Step 2)] \textbf{Reconstruction of moments from the regularised data: }
Let $\Psi^k $ be the reconstruction operator for smooth potentials, given by (\ref{ReOp}). 
The associated operator
\[
\tilde\Psi^k:  C_c^\infty  \Big((0,\infty) \times  T \Timk \mathbb{S}^{n-1} \Big)  \rightarrow C^\infty_c  \Big( (0,\infty) \times \Timk \R^n  \Big) , 
\] 
defined by
\[
\tilde\Psi^k f ( \epsilon, \bxk ) := [\Psi^k f  ( \epsilon , \cdot ) ] (\bxk), \ \bxk\in  \Timk \R^n, \ \epsilon\in (0,\infty)
\] 
transforms the regularised data into the corresponding family of moment maps, i.e.,
%, parametrised by regularisation parameter $\epsilon>0$: 
\[
 \tilde\Psi^k \Phi^k D^k ( \epsilon, \bxk ) = M^k_\epsilon (\bxk)
\]
\item[(Step 3)] \textbf{High-resolution limit: }
After the first two steps, the moment map $M^k \in \mathcal{E}' \big(\Timk \R^n\big)$ is obtained by taking the limit $\epsilon\longrightarrow 0$ within the space of distributions, that is,
\[
\lim_{\epsilon \rightarrow 0}  \Psi^k \Phi^k D^k ( \epsilon, \bxk ) = M^k (\bxk),
\]
\end{enumerate}

\subsection{Proof of Theorem \ref{mainmaintheorem}}
%Let  us denote  
%%\[
%%\R^\mathbb{N} := \prod_{j\in \mathbb{N} } 
%%\]
%\[
%\ell^2 := \{  (x_j)_{j=1}^\infty \in  \R^\mathbb{N} :  \|  (x_j)_{j=1}^\infty \|_{\ell^2} < \infty  \},
%\]
%\[
% \|  (x_j)_{j=1}^\infty \|_{\ell^2} := \Big( \sum_{j=1}^\infty |x_j|^2 \Big)^{1/2} ,
%\]
%where 
Let $\mathbb{R}^\mathbb{N}$ stand for the set of infinite sequences $(x_j)_{j=1}^\infty$, $x_j \in \R$, $j\in \mathbb{N}$ endowed with the smallest topology for which the coordinate  projections 
%$p_{j_0} : (x_j)_{j=1}^\infty \mapsto x_{j_0}$, $j_0 \in \mathbb{N}$ 
are continuous. We recall that $\R^\mathbb{N}$ is a Polish space. Also, let $\ell^2$ stand for the space of all sequences bounded in the $2$-norm.
Let $\iota : \ell^2 \hookrightarrow \R^\mathbb{N}$ stand for the trivial inclusion and define
\[
\iota^{-1} \mathcal{B} (\R^\mathbb{N})  %:= \sigma \{  \iota^{-1} B: \ B \in \mathcal{B} (\R^\mathbb{N})  \} 
:= \sigma \{  \ell^2 \cap B  : \ B \in \mathcal{B} (\R^\mathbb{N})  \}.
\]
%By uniqueness, 
%\[
%\iota_* (IW)_* \mathbb{P} = \mu^\mathbb{N} = \iota_* (IV)_{*} \mathbb{P},
%\]
%where $ \iota : \ell^2  \hookrightarrow \R^\mathbb{N}$ stands for the natural inclusion and $I : H^2 (\R^n) \rightarrow \ell^2$ is the isometry, given by $I  := (\langle \cdot, \varphi_j \rangle)_{j=1}^\infty$. 
%
%We need a simple lemma which shall be applied in the proof. 
Before continuing to the proof we record the following lemma:
%The following lemma provides a link between $\mu^\mathbb{N}$ and the probability distributions of the actual potential. 
%\begin{lemma}
\begin{lemma}\label{u9_lemma}
 The Borel algebra of $\ell^2$ is generated by $\iota$, that is,  $\mathcal{B} (\ell^2) 
= \iota^{-1}\mathcal{B} (\R^\mathbb{N})$.
\begin{proof}
Let $\mathcal{F}^j = (\pi^{\{j\}})^{-1} \mathcal{B} (\R)$ be the $\sigma$-algebra generated by the coordinate projection $\pi^{\{j\}} : \R^\mathbb{N} \rightarrow \R$, for each $j \in \mathbb{N}$.
Since the Borel algebra of $\R$ is generated by open sets, preimages of the form $(\pi^{\{j\}})^{-1} U$, where $U \subset \R$ is open, generate $\mathcal{F}^j$. 
As the union $ \bigcup_{j\in\mathbb{N} }\mathcal{F}_j$ generates $\mathcal{B} (\R^\mathbb{N})$, the algebra $\iota^{-1} \mathcal{B} (\R^\mathbb{N})$ is generated by the pre-images, $ \iota^{-1} (\pi^{\{j\}} )^{-1} U$, $j \in \mathbb{N}$, 
which are open in $\ell^2$ by continuity of the coordinate projections $(\pi^{\{j\}} \circ \iota)$, $j \in \mathbb{N}$. 
In particular, $\iota^{-1} \mathcal{B} (\R^\mathbb{N}) \subset \mathcal{B}(\ell^2)$.
 %By definition, $\pi^{\{j\}} $ is continuous on $\R^\mathbb{N}$ for every $j\in \mathbb{N}$. On the other hand, $\pi^{ \{j\} }$, $j\in \mathbb{N}$ is continuous also on $\ell^2$. Consequently, 
%Since the projections $\pi^{\{j\}}$,  $j\in \mathbb{N}$ generate the product $\sigma$-algebra and each composition $\pi^{\{j\}} \circ \iota$,  $j\in \mathbb{N}$ is continuous on $\ell^2$, we have $\iota^{-1}\mathcal{B} (\R^\mathbb{N}) \subset \mathcal{B} (\ell^2)$.
%
On the other hand, each closed ball $\overline{B}_r$ of $\ell^2$ with radius $r>0$, centered at $f \in \ell^2$, is the limit set
\[
\overline{B}_r = \bigcap_{m=1}^\infty  \Big\{ g \in \R^\mathbb{N} :  \sum_{j=1}^m   | f_j - \pi^{\{j\}} g  |^2  \in [0, r^2] \Big\}, \ f_j:= \pi^{\{j\}} \iota f
\]
%%the squared $\ell^2$-distance,
%%\[
%%\| \iota (f_j)_{j=1}^\infty  - ( \cdot ) \ \|_{\ell^2} 
%%\text{dist}_{\iota(f_j)_{j=1}^\infty } 
% %\mathcal{B} (\R^\mathbb{N}) \rightarrow \R : (g_j)_{j=1}^\infty \mapsto 
%% \min \left\{   \sum_{j=1}^\infty  | f_j - \pi^{\{j\}} |^2  , r^2 \right\}
%%\]
 which is measurable in $\R^\mathbb{N}$ as a countable intersection of measurable sets. 
Consequently, as every open set of a separable metric space is a countable union of closed balls, the Borel algebra of $\R^\mathbb{N}$ contains the topology of $\ell^2$ which implies $\mathcal{B} (\ell^2) \subset \iota^{-1}\mathcal{B} (\R^\mathbb{N})$.
  %Above, the coordinates of $f$ are denoted by $f_j:= \pi^{\{j\}} \iota f$, $j\in \mathbb{N}$.
% In summary, we conclude
% \begin{equation}\label{166EQQ}
% \mathcal{B} (\ell^2) = \iota^{-1}\mathcal{B} (\R^\mathbb{N}).
% \end{equation}
% %We conclude $\mathcal{B} (\ell^2) = \iota^{-1}\mathcal{B} (\R^\mathbb{N})$.
\end{proof}
\end{lemma}

We shall now continue to the proof. 
%It is well known that $l^2$ is a Hilbert space with 
Let the random field $V \in H^2( \mathbb{R}^n) \cap L^\infty (\R^n)$ satisfy
\begin{equation}\label{asdkh_100}
\mathbb{E}e^{a\|V\|_{H^2(\R^n)} } < \infty.
\end{equation}
Fix an orthonormal basis $\varphi_j$, $j\in \mathbb{N}$ of $H^2 (\mathbb{R}^n)$ and 
let $\mu^J (V_{j_1}, \dots, V_{j_k})$ be the distribution 
\[
%\begin{split}
%\begin{equation}\label{mitta}
\mu^J (V_{j_1}, \dots, V_{j_k}) := (V_{j_1}, \dots, V_{j_k})_* \mathbb{P} :   \mathcal{B} (\mathbb{R}^k) \rightarrow [0,1] , 
\]
with
%defined by
\[
   (V_{j_1}, \dots, V_{j_k})_* \mathbb{P}(B) := \mathbb{P} (  (V_{j_1}, \dots, V_{j_k}) \in B ),
%\end{split}
%\end{equation}
\]
generated by the finite collection of random variables $V_l := \langle V, \varphi_l \rangle$, $l \in J  = \{ j_1, \dots, j_k\} \subset \mathbb{N}$.
By definition, the moments satisfy
\[
\langle M^k , \varphi_{j_1} \otimes \cdots \otimes \varphi_{j_k} \rangle = \mathbb{E}(  \langle V , \varphi_{j_1} \rangle \cdots \langle V , \varphi_{j_k} \rangle  )= \mathbb{E}(  V_{j_1} \cdots V_{j_k}  )
\]
for every $j_1, \dots,j_k \in \mathbb{N}$. 
These quantities are determined from the data by Theorem \ref{thm:main}.  
Since,
\[
e^{a|(V_{j_1},\dots,V_{j_k})| } \leq e^{a \| (V_j)_{j=1}^\infty \|_{\ell^2} } = e^{a \| V \|_{H^2(\mathbb{R}^n)} },
%\leq e^{a|V_{j_1}| + \cdots + a|V_{j_k}| }  \leq e^{a\|V\|_{L^2}  \| \varphi_{j_1} \|_{L^2} +\cdots +a\|V\|_{L^2} \| \varphi_{j_k} \|_{L^2} }  = e^{ka\|V\|_{L^2} }
\]
the condition (\ref{asdkh_100})  yields
$
%\[
% \begin{equation}\label{suomiconditio}
%\int\limits_{\R^k} e^{a | (V_{j_1},\dots,V_{j_k})| } d \mu^J (V_1, \dots, V_k) 
%=
 \mathbb{E } e^{a | (V_1,\dots,V_k)| } 
< \infty,
%\end{equation}
%\]
$
 implying that the moments, and hence the data uniquely describes each measure $\mu^J (V_1, \dots, V_k)$ with finite $J \subset \mathbb{N}$, as shown in \cite{DVURECENSKIJ200255}.
 
 Suppose that $W \in H^2 (\mathbb{R}^n)$ is a random field that satisfies  (\ref{asdkh_100}) and yields the same data as $V$. 
In particular, 
\[
\mu^J (V_{j_1}, \dots, V_{j_k} ) = \mu^J (W_{j_1}, \dots, W_{j_k} ),
\]
 for every finite $J = \{ j_1, \dots, j_k\}$. 
 %We shall omit to  the random vectors and denote simply $\mu^J$. %We shall denote the measure by $
%Equipping the product space $\mathbb{R}^\mathbb{N}$ with the metric 
% \[
%d\big(  (x_j)_1^\infty , (y_j)_1^\infty \big) = \sum_{j=1}^\infty \frac{1}{2^j} \frac{|x_j - y_j |}{1 +|x_j - y_j  | }.
% \]
%results a Polish space whose Borel algebra equals to the associated product algebra.
%We recall that a countable product of Polish spaces is again a polish space
It is a consequence of the Kolmogorov extension theorem \cite[Thm. 6.16]{kallenberg2006foundations} that there is an unique probability distribution 
\[
\mu^\mathbb{N}  :
% \prod_{\mathbb{N}} 
\mathcal{B}(\R^\mathbb{N} ) \rightarrow [0,1] 
\]
in $\R^{\mathbb{N}}$, 
satisfying $  \mu^J(V_{j_1}, \dots, V_{j_k} )  =\pi^{J}_* \mu^\mathbb{N}$ for any finite $J \subset \mathbb{N}$ where $\pi^J : \R^\mathbb{N} \rightarrow \R^{|J|}$ is the natural projection. 
%Further, the distribution is the limit of $\mu^{\{ 1,\dots,k \}}$ as $k\longrightarrow \infty$.
%Since the inclusion $\ell^2 (\mathbb{R} ) \hookrightarrow \R^\mathbb{N}$ is continuous, the distribution

Finally, we can show that the probability distributions $V_* \mathbb{P}$ and $W_*\mathbb{P}$ equal as measures on Borel sets of $\ell^2$.  
Let $I : H^2 (\R^n) \rightarrow \ell^2$ be the isometry $I  := (\langle \cdot, \varphi_j \rangle)_{j=1}^\infty$. 
It enough to prove the claim for elements $I(V)$ and $I(W)$. 
By uniqueness of $\mu^\mathbb{N}$,
\[
\iota_*  I(V)_* \mathbb{P} = \mu^\mathbb{N} = \iota_*  I(W)_* \mathbb{P} ,
\]
that is,
\[
I(V)_* \mathbb{P}( \iota^{-1} A )= I(W)_* \mathbb{P} ( \iota^{-1} A )
\]
for every $A \in \mathcal{B} (\R^\mathbb{N})$. 
Thus we obtain  
$
%I_* V_* \mathbb{P} 
 I ( V)_* \mathbb{P} = I ( W)_* \mathbb{P}
%= I_* W_* \mathbb{P}
$ 
 by Lemma \ref{u9_lemma}. This completes the proof of Theorem \ref{mainmaintheorem}.
 
 %Since $I$ is an isometry between Borel-spaces.

\section{Appendix: Exterior Data on $\Timk T\mathbb{S}^{n-1}$}

We shall show in detail how the representation of the exterior data $D^k(\bxk,\btk)$ on the tangent bundle $T \prod_{j=1}^k \mathbb{S}^{n-1}$ is constructed. Below, $\{ \theta \}^\perp := \{ x\in \R^n : x \cdot \theta = 0 \}$ for $\theta \in \S^{n-1}$. 
\begin{lemma}
%Define $\text{pr}_2 ( \theta ,y) = y$. 
The vector bundles $T \mathbb{S}^{n-1}$ and $  \bigcup_{\theta \in \mathbb{S}^{n-1} }   \{ \theta \}^\perp \times  \{ \theta\}  \rightarrow  \mathbb{S}^{n-1}$ with the trivial projection $\text{pr}_2 ( v, \theta) := \theta$  are isomorphic.  The isomorphism is obtained from the tangent of the inclusion $\mathbb{S}^{n-1} \hookrightarrow \R^{n}$ by restricting the codomain. 
\begin{proof}
By expressing each tangent vector $v \in T_\theta \mathbb{S}^{n-1}$ as an equivalence class of curves $\gamma : (-\epsilon,\epsilon) \rightarrow \mathbb{S}^{n-1}$ with $\gamma(0) = \theta$, $\dot\gamma (0) = v$, the tangent map $T\iota : T \mathbb{S}^{n-1} \hookrightarrow T \R^n = \R^n \times \R^n$ associated to the inclusion $\iota : \mathbb{S}^{n-1} \hookrightarrow \R^n$ takes the form $T\iota [\gamma] = [\iota \circ \gamma]$. As $( \iota \circ \gamma )' (0) $ is perpendicular to the base point $\iota \circ \gamma(0)$, we obtain $T\iota (T_\theta \mathbb{S}^{n-1}  ) \subset \{ \theta \}^{\perp} \times  \{ \theta \} $  which, together with injectivity and the rank-nullity theorem, implies that $T \iota$ defines a bundle isomorphism %between $\bigcup_{\theta \in \mathbb{S}^{n-1} } \{ \theta \} \times  \{ \theta\}^\perp \rightarrow \mathbb{S}^{n-1}$.
\[
\xymatrix{
T \mathbb{S}^{n-1} \ar[dr]\ar[r]^<<<{T \iota}  &  \bigcup_{\theta \in \mathbb{S}^{n-1} }  \{ \theta \}^\perp \times  \{ \theta\}  \ar[d]^{\text{pr}_2} \\
& \mathbb{S}^{n-1} 
}
\]
Note the untypical order of the base point $\theta$ and the fiber space $\{ \theta \}^\perp$ in $ \{ \theta \}^\perp \times  \{ \theta\} $. 
\end{proof}
\end{lemma}

The exterior data can be rewritten as a generalised function $ \btk \mapsto  D^k(\btk)  \in \mathcal{D}' ( T_{\btk} \prod_{j=1}^k \mathbb{S}^{n-1} )$ of the bundle $T \prod_{j=1}^k \mathbb{S}^{n-1} $ by setting
\[
\langle D^k (\btk)  ,  \varphi  \rangle =  \Big\langle D^k   ,   \varphi \circ   (T\iota)^{-1} |_{\Timk \{ \theta_j \}^\perp \times \{ \theta_j \} }  \Big\rangle,  \ \varphi \in C_c^\infty \left(T_{\btk} \prod_{j=1}^k  \mathbb{S}^{n-1} \right)
\]
where---with abuse of notation---we denote by $T\iota$ also the associated isomorphism 
\[
\prod_{j=1}^k T \mathbb{S}^{n-1} \rightarrow \prod_{j=1}^k  \bigcup_{\theta \in \mathbb{S}^{n-1} }  \{ \theta \}^\perp \times  \{ \theta\},
\]
 generated by the tangent map above. It is a consequence of the invariance (\ref{shift}) that no information is lost in the identification. That is, the lift of $D^k ( \cdot , \btk)$ to the space $T_{\btk} \Timk \S^{n-1} = \Timk \{ \theta_j \}^\perp \times \{ \theta_j \}$ is well defined and unique for each $\btk \in \Timk \S^{n-1}$.

%\[
%\iota_* : \prod_{j=1}^k T \mathbb{S}^{n-1} \rightarrow \prod_{j=1}^k  \bigcup_{\theta \in \mathbb{S}^{n-1} }  \{ \theta \}^\perp \times  \{ \theta\}, 
%\] 
%It is due to the symmetry property (\ref{shift}) that
%\[
%D_{\btk} 
%\]

% \bibliographystyle{abbrv}
%\bibliographystyle{plain}
\bibliographystyle{acm}
\bibliography{bibs_scattering,citations,one_more,mattis_extras}

\begin{thebibliography}{10}

\bibitem{bal2010kinetic}
{\sc Bal, G., Komorowski, T., and Ryzhik, L.}
\newblock Kinetic limits for waves in a random medium.
\newblock {\em Kinetic and Related Models 3}, 4 (2010), 529--644.

\bibitem{Bao2016several}
{\sc Bao, G., Chen, C., and Li, P.}
\newblock Inverse random source scattering problems in several dimensions.
\newblock {\em SIAM/ASA J. Uncertain. Quantif. 4}, 1 (2016), 1263--1287.

\bibitem{bao2017inverse}
{\sc Bao, G., Chen, C., and Li, P.}
\newblock Inverse random source scattering for elastic waves.
\newblock {\em SIAM Journal on Numerical Analysis 55}, 6 (2017), 2616--2643.

\bibitem{Bao2014Helmholtz}
{\sc Bao, G., Chow, S.-N., Li, P., and Zhou, H.}
\newblock An inverse random source problem for the {H}elmholtz equation.
\newblock {\em Math. Comp. 83}, 285 (2014), 215--233.

\bibitem{MR3606418}
{\sc Borcea, L., and Kocyigit, I.}
\newblock Imaging in random media with convex optimization.
\newblock {\em SIAM J. Imaging Sci. 10}, 1 (2017), 147--190.

\bibitem{borcea2003}
{\sc Borcea, L., Papanicolaou, G., and Tsogka, C.}
\newblock Theory and applications of time reversal and interferometric imaging.
\newblock {\em Inverse Problems 19}, 6 (2003), S139.

\bibitem{MR2249471}
{\sc Borcea, L., Papanicolaou, G., and Tsogka, C.}
\newblock Adaptive interferometric imaging in clutter and optimal illumination.
\newblock {\em Inverse Problems 22}, 4 (2006), 1405--1436.

\bibitem{MR1943391}
{\sc Borcea, L., Papanicolaou, G., Tsogka, C., and Berryman, J.}
\newblock Imaging and time reversal in random media.
\newblock {\em Inverse Problems 18}, 5 (2002), 1247--1279.

\bibitem{caro2016inverse}
{\sc Caro, P., Helin, T., and Lassas, M.}
\newblock Inverse scattering for a random potential.
\newblock {\em arXiv preprint arXiv:1605.08710\/} (2016).

\bibitem{zbMATH00939857}
{\sc {Colton}, D., and {Kirsch}, A.}
\newblock {A simple method for solving inverse scattering problems in the
  resonance region.}
\newblock {\em {Inverse Probl.} 12}, 4 (1996), 383--393.

\bibitem{dehoop09}
{\sc De~Hoop, M.~V., and Solna, K.}
\newblock Estimating a green's function from “field-field” correlations in
  a random medium.
\newblock {\em SIAM Journal on Applied Mathematics 69}, 4 (2009), 909--932.

\bibitem{devaney1979inverse}
{\sc Devaney, A.}
\newblock The inverse problem for random sources.
\newblock {\em Journal of Mathematical Physics 20}, 8 (1979), 1687--1691.

\bibitem{duistermaat2010fourier}
{\sc Duistermaat, J.}
\newblock {\em Fourier Integral Operators}.
\newblock Modern Birkh{\"a}user Classics. Birkh{\"a}user Boston, 2010.

\bibitem{DVURECENSKIJ200255}
{\sc Dvurečenskij, A., Lahti, P., and Ylinen, K.}
\newblock The uniqueness question in the multidimensional moment problem with
  applications to phase space observables.
\newblock {\em Reports on Mathematical Physics 50}, 1 (2002), 55 -- 68.

\bibitem{MR1012864}
{\sc Eskin, G., and Ralston, J.}
\newblock The inverse backscattering problem in three dimensions.
\newblock {\em Comm. Math. Phys. 124}, 2 (1989), 169--215.

\bibitem{MR1110451}
{\sc Eskin, G., and Ralston, J.}
\newblock Inverse backscattering in two dimensions.
\newblock {\em Comm. Math. Phys. 138}, 3 (1991), 451--486.

\bibitem{fouque2007wave}
{\sc Fouque, J.-P., Garnier, J., Papanicolaou, G., and Solna, K.}
\newblock {\em Wave propagation and time reversal in randomly layered media},
  vol.~56.
\newblock Springer Science \& Business Media, 2007.

\bibitem{friedlander1998introduction}
{\sc Friedlander, F., and Joshi, M.}
\newblock {\em Introduction to the Theory of Distributions}.
\newblock Cambridge University Press, 1998.

\bibitem{garnier2009passive}
{\sc Garnier, J., and Papanicolaou, G.}
\newblock Passive sensor imaging using cross correlations of noisy signals in a
  scattering medium.
\newblock {\em SIAM Journal on Imaging Sciences 2}, 2 (2009), 396--437.

\bibitem{garnier2016passive}
{\sc Garnier, J., and Papanicolaou, G.}
\newblock {\em Passive imaging with ambient noise}.
\newblock Cambridge University Press, 2016.

\bibitem{MR2482156}
{\sc Garnier, J., and S\o~lna, K.}
\newblock Background velocity estimation with cross correlations of incoherent
  waves in the parabolic scaling.
\newblock {\em Inverse Problems 25}, 4 (2009), 045005, 34.

\bibitem{MR3318382}
{\sc Garnier, J., and S\o~lna, K.}
\newblock Transmission and reflection of electromagnetic waves in randomly
  layered media.
\newblock {\em Commun. Math. Sci. 13}, 3 (2015), 707--728.

\bibitem{helin2018atmospheric}
{\sc Helin, T., Kindermann, S., Lehtonen, J., and Ramlau, R.}
\newblock Atmospheric turbulence profiling with unknown power spectral density.
\newblock {\em Inverse Problems\/} (2018).

\bibitem{helin2016correlation}
{\sc Helin, T., Lassas, M., Oksanen, L., and Saksala, T.}
\newblock Correlation based passive imaging with a white noise source.
\newblock {\em arXiv preprint arXiv:1609.08022\/} (2016).

\bibitem{HLP}
{\sc Helin, T., Lassas, M., and P\"aiv\"arinta, L.}
\newblock Inverse acoustic scattering problem in half-space with anisotropic
  random impedance.
\newblock {\em arxiv:1407.2481\/} (2015).

\bibitem{MR3325344}
{\sc Hu, G., Li, J., and Liu, H.}
\newblock Uniqueness in determining refractive indices by formally determined
  far-field data.
\newblock {\em Appl. Anal. 94}, 6 (2015), 1259--1269.

\bibitem{ishimaru1978wave}
{\sc Ishimaru, A.}
\newblock {\em Wave propagation and scattering in random media}, vol.~2.
\newblock Academic press New York, 1978.

\bibitem{kallenberg2006foundations}
{\sc Kallenberg, O.}
\newblock {\em Foundations of modern probability}.
\newblock Springer Science \& Business Media, 2006.

\bibitem{MR1491678}
{\sc Kurylev, Y., and Starkov, A.}
\newblock Directional moments in the acoustic inverse problem.
\newblock In {\em Inverse problems in wave propagation ({M}inneapolis, {MN},
  1995)}, vol.~90 of {\em IMA Vol. Math. Appl.} Springer, New York, 1997,
  pp.~295--323.

\bibitem{MR1991787}
{\sc Kurylev, Y.~V., Mandache, N., and Peat, K.~S.}
\newblock Hausdorff moments in an inverse problem for the heat equation:
  numerical experiment.
\newblock {\em Inverse Problems 19}, 2 (2003), 253--264.

\bibitem{MR1474373}
{\sc Kurylev, Y.~V., and Peat, K.~S.}
\newblock Hausdorff moments in two-dimensional inverse acoustic problems.
\newblock {\em Inverse Problems 13}, 5 (1997), 1363--1377.

\bibitem{LPS}
{\sc Lassas, M., P{\"a}iv{\"a}rinta, L., and Saksman, E.}
\newblock Inverse scattering problem for a two dimensional random potential.
\newblock {\em Comm. Math. Phys. 279}, 3 (2008), 669--703.

\bibitem{li2017inverse}
{\sc Li, M., Chen, C., and Li, P.}
\newblock Inverse random source scattering for the helmholtz equation in
  inhomogeneous media.
\newblock {\em Inverse Problems 34}, 1 (2017), 015003.

\bibitem{Li2011source}
{\sc Li, P.}
\newblock An inverse random source scattering problem in inhomogeneous media.
\newblock {\em Inverse Problems 27}, 3 (2011), 035004, 22.

\bibitem{li2017stability}
{\sc Li, P., and Yuan, G.}
\newblock Stability on the inverse random source scattering problem for the
  one-dimensional helmholtz equation.
\newblock {\em Journal of Mathematical Analysis and Applications 450}, 2
  (2017), 872--887.

\bibitem{MR2582602}
{\sc Liu, H., Zhang, H., and Zou, J.}
\newblock Recovery of polyhedral scatterers by a single electromagnetic
  far-field measurement.
\newblock {\em J. Math. Phys. 50}, 12 (2009), 123506, 10.

\bibitem{lohwater2013boundary}
{\sc Lohwater, J., and Ladyzhenskaya, O.}
\newblock {\em The Boundary Value Problems of Mathematical Physics}.
\newblock Applied Mathematical Sciences. Springer New York, 2013.

\bibitem{zbMATH04105476}
{\sc {Nachman}, A.~I.}
\newblock {Reconstructions from boundary measurements.}
\newblock {\em {Ann. Math. (2)} 128}, 3 (1988), 531--576.

\bibitem{zbMATH04129351}
{\sc {Novikov}, R.}
\newblock {Multidimensional inverse spectral problem for the equation $-\Delta
  \psi -(v(x)-Eu(x))\psi =0$.}
\newblock {\em {Funct. Anal. Appl.} 22}, 4 (1988), 263--272.

\bibitem{petkov1989scattering}
{\sc Petkov, V.}
\newblock {\em Scattering Theory for Hyperbolic Operators}.
\newblock Studies in Mathematics and its Applications. Elsevier Science, 1989.

\bibitem{MR2384771}
{\sc Rakesh}.
\newblock Inverse problems for the wave equation with a single coincident
  source-receiver pair.
\newblock {\em Inverse Problems 24}, 1 (2008), 015012, 16.

\bibitem{MR3224125}
{\sc Rakesh, and Uhlmann, G.}
\newblock Uniqueness for the inverse backscattering problem for angularly
  controlled potentials.
\newblock {\em Inverse Problems 30}, 6 (2014), 065005, 24.

\bibitem{zbMATH04028038}
{\sc {Ramm}, A.}
\newblock {Recovery of the potential from fixed energy scattering data.}
\newblock {\em {Inverse Probl.} 4}, 3 (1988), 877--886.

\bibitem{MR3470117}
{\sc Scherzer, O.}, Ed.
\newblock {\em Handbook of mathematical methods in imaging. {V}ol. 1, 2, 3},
  second~ed.
\newblock Springer, New York, 2015.

\bibitem{schmudgen2017moment}
{\sc Schm{\"u}dgen, K.}
\newblock {\em The Moment Problem}.
\newblock Graduate Texts in Mathematics. Springer International Publishing,
  2017.

\bibitem{MR1708341}
{\sc Shevtsov, B.~M.}
\newblock Backscattering and inverse problem in random media.
\newblock {\em J. Math. Phys. 40}, 9 (1999), 4359--4373.

\bibitem{10.2307/2374463}
{\sc Shiota, T.}
\newblock An inverse problem for the wave equation with first order
  perturbation.
\newblock {\em American Journal of Mathematics 107}, 1 (1985), 241--251.

\end{thebibliography}

%\section{Comments here}
%
%\begin{enumerate}
%
%
%\item these sources to the intro: 
%\newline \textbullet[Eskin1960, A sufficient condition for the solvability of a multidimensional problem
%of moments, Dokl- Akad. Nauk SSSR 133, 540?543 (1960)]
%\newline
%\textbullet [B. Fuglede1985: The multidimensional moment problem. Expo. Math. 1, 47?65 (1983)]
%\newline
%\textbullet [ A. E. Nussbaum1965: Quasi-analytic vectors, Ark. Math. 6, 179?191 (1965)] 
%
%
%
%%\item Probabilistic criteria for $V(x) \in C_{comp}$ and $a_1 \in H^2_{loc}$ ?
%
%%\item frequency domain approach
%
%%\item Introduction, styles, etc.
%
%\end{enumerate}

\end{document}